\documentclass[12pt,a4paper,numbers=endperiod]{scrartcl}
\usepackage[utf8]{inputenc}
\usepackage[T1]{fontenc}
\usepackage{amsmath}
\usepackage{amssymb}
\usepackage{mathtools}
\usepackage{enumerate} 
\usepackage{tikz-cd}
\usepackage{leftidx}
\usepackage[colorlinks=true, linkcolor=red, citecolor=blue, urlcolor=cyan]{hyperref}
\usepackage{amsthm}
\usepackage{thmtools}
\usepackage{thm-restate}
\usepackage{geometry}
\usepackage{cleveref} 
\usepackage{url}

\RedeclareSectionCommand[tocnumwidth=1.75em]{section}

\makeatletter
\renewcommand*\env@matrix[1][*\c@MaxMatrixCols c]{%
  \hskip -\arraycolsep
  \let\@ifnextchar\new@ifnextchar
  \array{#1}}
\makeatother

\DeclareSymbolFont{bbold}{U}{bbold}{m}{n}
\DeclareSymbolFontAlphabet{\mathbbold}{bbold}

\makeatletter
\g@addto@macro\bfseries{\boldmath}
\makeatother

\newcommand{\comment}[1]{} 

\newcommand{\IN}{\mathbb{N}}
\newcommand{\IP}{\mathbb{P}}
\newcommand{\IQ}{\mathbb{Q}}

\newcommand{\IZ}{\mathbb{Z}}

\newcommand{\inv}{^{-1}}
\newcommand{\tensor}{\otimes}
\newcommand{\mf}[1]{\mathfrak{#1}}
\newcommand{\mc}[1]{\mathcal{#1}}
\newcommand{\ve}{\varepsilon}

\newcommand{\units}[1]{U(#1)}
\renewcommand{\d}{\partial}

\DeclarePairedDelimiter{\abs}{\lvert}{\rvert}

\DeclarePairedDelimiter{\erz}{\langle}{\rangle}
\DeclarePairedDelimiter{\merz}{[}{]}
\newcommand{\nset}[2][1]{\set{#1,\ldots,#2}}

\newcommand{\set}[1]{\left\{#1\right\}}

\newcommand{\implication}[2]{\item[\normalfont\textit{(#1)}$\Rightarrow$\textit{(#2)}:]}

\DeclareMathOperator{\LSat}{LSat}
\DeclareMathOperator{\im}{im}

\DeclareMathOperator{\Quot}{Quot}
\DeclareMathOperator{\Tor}{Tor}
\DeclareMathOperator{\chara}{char}
\DeclareMathOperator{\GKdim}{GKdim}
\DeclareMathOperator{\trdeg}{tr.deg}

\let\PROOF=\proof
\renewcommand\proof{\PROOF[\textbf{Proof:}]} 

\declaretheorem[name=Theorem,numberwithin=section]{thm}
\declaretheorem[name=Proposition,sibling=thm,refname={proposition,propositions}]{prop}
\declaretheorem[name=Lemma,sibling=thm]{lemma}
\declaretheorem[name=Corollary,sibling=thm]{cor}
\declaretheorem[name=Definition,sibling=thm,style=definition]{definition}
\declaretheorem[name=Remark,sibling=thm,style=definition]{rem}

\declaretheorem[name=Example,sibling=thm,style=definition]{ex}

\let\MOD=\mod
\renewcommand{\mod}[1]{\!\!\MOD{#1}}

\title{Left saturation closure for Ore localizations}
\author{Johannes Hoffmann and Viktor Levandovskyy}

\begin{document}

\maketitle

\tableofcontents

\addsec{Abstract}

In this paper, we introduce the notion of $\LSat$, the left saturation closure of a subset of a module at a subset of the base ring, which generalizes multiple important concepts related to Ore localization.
We show its significance in finding a saturated normal form for left Ore sets as well as in characterizing the units of a localized ring.

Furthermore, $\LSat$ encompasses the notion of local closure of submodules and ideals from the realm of algebraic analysis, where it describes the result of extending a submodule or ideal from a ring to its localization and contracting it back again.

\section{Introduction}

Localizing a commutative ring $R$ at a multiplicative set $S$ is an important and well-understood tool in the study of commutative rings.
For non-commutative domains, the concept of Ore localization introduced by {\O}ystein Ore (\cite{ore}) is a generalization that retains most of the properties of classical localization at the cost of additional requirements for the set $S$: apart from being a multiplicative set, we need $S$ to satisfy the left Ore condition, that is, for any $r\in R$ and $s\in S$ we need to have $Sr\cap Rs\neq\emptyset$.

The starting point for our work was the following phenomenon that already appears in the commutative setting: consider the polynomial ring $R=K[x]$ over a field $K$ and the multiplicative set $S$ consisting of all non-negative powers of $x^2$.
In the localization $S\inv R$, we have the element $\frac{x}{x^2}$, which, intuitively, should be the same as $\frac{1}{x}$.
But the latter is not a well-defined object in $S\inv R$ since $x\notin S$.
To solve this problem we can consider the set $\tilde{S}$ of all non-negative powers of $x$ and the associated localization $\tilde{S}\inv R$, where we indeed have $\frac{x}{x^2}=\frac{1}{x}$.
It is easy to see that the localizations $S\inv R$ and $\tilde{S}\inv R$ are isomorphic, but the latter one seems to be a better description of the actual denominators occurring in the localization.

This gave rise to the following question: given a left Ore set $S$ in a (non-commutative) domain $R$, is there a standardized representation of the localization $S\inv R$ via another Ore set $\tilde{S}\supseteq S$ such that $S\inv R\cong\tilde{S}\inv R$ and the phenomenon above does not manifest in $\tilde{S}\inv R$?
From the example we can already read off a necessary condition: the set $\tilde{S}$ should be saturated in the sense that if $st\in\tilde{S}$, then also $s,t\in\tilde{S}$.

For commutative rings, the answer to this question has been known for a while: Zariski and Samuel (\cite{zariski_samuel}) showed that the set $\tilde{S}$ consisting of all elements of $R$ that divide an element of $S$ satisfies the requirements stated above.
Moreover, they noted that the units in $S\inv R$ are exactly the fractions with a numerator from $\tilde{S}$, which shows that our original question is connected to the problem of identifying the units of the localization.

More recently, Bavula (\cite{bavula}) demonstrated that taking the inverse image of the unit group of the localization under the canonical embedding of the base ring yields a saturated Ore set that describes the same localization.
While this essentially answers the question stated above, we were interested in a more constructive description of this set and set out to uncover the theory behind it.

In this paper, after briefly recalling the construction and some important properties of Ore localization, we introduce the general concept of left saturation, more precisely, we define $\LSat_T(P)$, the left $T$-saturation of $P$, where $T$ is a subset of $R$ and $P$ is a subset of a left $R$-module $M$.
This rather technical notion encompasses at least two important concepts connected to Ore localization that emerge when we specialize the parameters $T$ and $P$.

The first application considers $\LSat(S):=\LSat_R(S)$, the left saturation closure of a left Ore set $S$.
We show that $\LSat(S)$ indeed has the desired properties stated above, in particular, given $S$, $\LSat(S)$ gives us full insight into the structure of the localization $S\inv R$ as we have a complete description of the denominators as well as the units.

The second application arises when we consider $P^S:=\LSat_S(P)$ for a left Ore set $S$ and a submodule $P$.
Then $P^S$ is called (left) $S$-closure (or local closure) of $P$.
This object appears prominently in algebraic analysis, where it describes the result of extending a submodule or ideal from a ring to its localization and contracting it back again.

Throughout this paper, we will give examples to illustrate the concepts.

\section{Preliminaries}

All rings are assumed to be associative, unital and non-trivial (i.e., not the zero ring).
We will need the following properties:

\begin{definition}
	Let $S$ be a subset of a ring $R$.
	We call $S$
	\begin{itemize}
		\item
			a \emph{multiplicative set} if $1_R\in S$, $0_R\notin S$ and for all $s,t\in S$ we have $st\in S$.
		\item
			\emph{left saturated} (resp. \emph{right saturated}) if for all $s,t\in S$, $st\in S$ implies $t\in S$ (resp. $s\in S$).
		\item
			\emph{saturated} if $S$ is both left and right saturated.
	\end{itemize}
\end{definition}

Recall that a ring is called a \emph{domain} if for all $a,b\in R$, $a\cdot b=0$ implies $a=0$ or $b=0$.
Any subset (not containing zero) of a domain has a multiplicative superset which is minimal with respect to inclusion:

\begin{definition}
	Let $R$ be a domain.
	The \emph{multiplicative closure} or \emph{monoid closure} of a subset $M$ of $R\setminus\set{0}$ is $\merz{M}:=\set{\prod_{i=1}^{n}m_i\mid n\in\IN_0,m_i\in M}$, where the empty product equals $1_R$.
\end{definition}

To decide whether one multiplicative set is contained in another, it suffices to check this on generators:

\begin{lemma}\label{inclusions_of_multiplicatively_closed_sets_on_generators}
	Let $S,T$ be multiplicative sets in a ring $R$.
	Furthermore, let $I$ be an arbitrary index set and $s_i\in S$ for all $i\in I$ such that $S=\merz{\set{s_i\mid i\in I}}$.
	Then $S\subseteq T$ if and only if $s_i\in T$ for all $i\in I$.
\end{lemma}

Recall that an element $u$ of a ring $R$ is called \emph{invertible} or a \emph{unit} if there exists $v\in R$ such that $uv=1=vu$.
The set $U(R)$ of all units of $R$ forms a multiplicative group and is therefore a multiplicative set in $R$.

\begin{lemma}\label{unit_group_is_saturated}
	Let $R$ be a domain.
	Then $U(R)$ is saturated.
\end{lemma}
\begin{proof}
	Let $a,b\in R$ such that $a\cdot b\in U(R)$.
	Then there exists $u\in U(R)$ such that $abu=1=uab$, thus $a$ is right-invertible and $b$ is left-invertible.
	Since $R$ is a domain this implies that $a,b\in U(R)$.
\end{proof}

In a commutative domain, we can localize at any multiplicative subset, whereas in the non-commutative setting, we additionally need the Ore condition:

\begin{definition}
	Let $S$ be a subset of a domain $R$.
	\begin{itemize}
		\item
			We say that $S$ satisfies the \emph{left Ore condition} in $R$ if for all $s\in S$ and $r\in R$ there exist $\tilde{s}\in S$ and $\tilde{r}\in R$ such that $\tilde{s}r=\tilde{r}s$.
		\item
			We call $S$ a \emph{left Ore set} in $R$ if it is a multiplicative set and satisfies the left Ore condition in $R$.
	\end{itemize}
\end{definition}

If $S$ satisfies the left Ore condition, then any finite selection of elements from $S$ has a common left multiple in $S$.
This is utilized to find common denominators when localizing at $S$.

\begin{rem}
	Let $K$ be a field and $\mc{D}=K\erz{x,\partial\mid\partial x=x\partial+1}$ the \emph{first Weyl algebra} over $K$.
	This Noetherian domain will be the main source of examples for this paper, because already in this mildly non-commutative situation we can observe all occurring phenomena.
	In $\mc{D}$, we work with the \emph{Euler operator} $\theta:=x\partial$ and the $\IZ$-grading induced by setting $\deg(x)=-1$ and $\deg(\partial)=1$.
\end{rem}

\begin{ex}\label{examples_of_left_Ore_sets_in_first_weyl_algebra}
	Consider the following sets in $\mc{D}$:
	\begin{itemize}
		\item
			For any $f\in K[x]\cup K[\partial]$, $\merz{f}$ is a left Ore set (\cite{HL17}, Lemma 4.3).
		\item
			$\merz{\theta}=\merz{x\partial}$ is not a left Ore set, since for any $l\in\IN_0$, by the forthcoming \Cref{theta_proposition} we have that $\theta^lx=x(\theta+1)^l$, which can never be of the form $f\theta^k$ for any $f\in\mc{D}$ and a given $k\in\IN$.
		\item
			If $\chara(K)=0$, then $\Theta:=\merz{\theta+\IZ}$ is a left Ore set in $\mc{D}$, but neither left nor right saturated.
			The last part is obvious from the fact that $x\partial=\theta\in\Theta$, but $x\notin\Theta$ as well as $\partial\notin\Theta$.
			The first part will be proven in multiple steps in the forthcoming \Cref{Theta_is_left_Ore}.
	\end{itemize}
\end{ex}

Given two left Ore sets $S,T$ in $R$, their product $ST:=\set{st\mid (s,t)\in S\times T}$ satisfies the left Ore condition, but is not a multiplicative set in general (\cite{skoda_2006}, 6.9).
The multiplicative closure of the product, however, retains the left Ore property:

\begin{lemma}\label{merz_of_union_of_Ore_sets_is_Ore}
	Let $R$ be a domain and $J$ a non-empty index set such that $S_j\subseteq R$ is a left Ore set in $R$ for every $j\in J$.
	Then $T:=[\bigcup_{j\in J}S_j]$ is a left Ore set in $R$.
\end{lemma}
\begin{proof}
	The proof works by induction over the number of factors from the sets $S_j$ that appear in an element of $T$.
	Details can be found in \cite{krause_lenagan}. 
\end{proof}

\begin{ex}\label{merz_of_x_and_d_is_left_Ore}
	The set $S=\merz{x,\partial}=\merz{\merz{x}\cup\merz{\partial}}$ is a left Ore set in $\mc{D}$ by \Cref{examples_of_left_Ore_sets_in_first_weyl_algebra} and \Cref{merz_of_union_of_Ore_sets_is_Ore}, but neither left nor right saturated: we have $\partial(x\partial-1)=x\partial^2\in S$ and $(x\partial-1)x=x^2\partial\in S$, but $x\partial-1$ is not contained in $S$.
\end{ex}

\section{Ore localization of domains}

In this section we recall the most important notions and known results of the technique of Ore localization.

\begin{definition}\label{definition_left_ore_localization}
	Let $S$ be a multiplicative set in a domain $R$.
	A ring $R_S$ together with an injective homomorphism $\varphi:R\rightarrow R_S$ is called a \emph{left Ore localization} of $R$ at $S$ if:
	\begin{enumerate}[(i)]
		\item
			For all $s\in S$, $\varphi(s)$ is a unit in $R_S$.
		\item
			For all $x\in R_S$, we have $x=\varphi(s)\inv\varphi(r)$ for some $s\in S$ and $r\in R$.
	\end{enumerate}
\end{definition}

It can be shown that a left Ore localization of $R$ at $S$ exists if and only if $S$ is a left Ore set in $R$; in this case the localization is unique up to isomorphism.
The following gives an explicit construction:

\begin{thm}[Ore, 1931]\label{thm_construction}
	Let $S$ be a left Ore set in a domain $R$ and $S\inv R:=(S\times R)/\sim$, where the equivalence relation $\sim$ is given by
	\[
		(s_1,r_1)\sim(s_2,r_2)\quad\Leftrightarrow\quad
		\exists~\tilde{s}\in S,\exists~\tilde{r}\in R:\tilde{s}s_2=\tilde{r}s_1\text{ and }\tilde{s}r_2=\tilde{r}r_1.
	\]
	Together with the operations
	\[
		+:S\inv R\times S\inv R\rightarrow S\inv R,\quad
		(s_1,r_1)+(s_2,r_2):=(\tilde{s}s_1,\tilde{s}r_1+\tilde{r}r_2),
	\]
	where $\tilde{s}\in S$ and $\tilde{r}\in R$ satisfy $\tilde{s}s_1=\tilde{r}s_2$, and
	\[
		\cdot:S\inv R\times S\inv R\rightarrow S\inv R,\quad
		(s_1,r_1)\cdot(s_2,r_2):=(\tilde{s}s_1,\tilde{r}r_2),
	\]
	where $\tilde{s}\in S$ and $\tilde{r}\in R$ satisfy $\tilde{s}r_1=\tilde{r}s_2$, $(S\inv R,+,\cdot)$ becomes a ring with $0_{S\inv R}=(1_R,0_R)$ and $1_{S\inv R}=(1_R,1_R)$.
\end{thm}
\begin{proof}
	Checking the ring axioms is a rather lengthy task that heavily involves the Ore condition and is generally avoided in the literature.
	For (mostly) complete proofs we refer to \cite{skoda_2006} and the original paper \cite{ore}.
\end{proof}

By abuse of notation we denote the elements of $S\inv R$ again by $(s,r)$.
The restriction $0\notin S$ included in the definition of multiplicative sets exists to avoid trivialities, since $0\in S$ holds if and only if $S\inv R=\set{0}$.

The equivalence relation $\sim$ defined in \Cref{thm_construction} is inspired by \cite{skoda_2006}.
At first glance, it seems to be different from
\[
	(s_1,r_1)\approx(s_2,r_2)
	\quad:\Leftrightarrow\quad
	\exists~a,b\in R:
	as_2=bs_1\in S\text{ and }ar_2=br_1,
\]
which is the one given in many textbooks (e.g., \cite{BGV,goodearlwarfield,krause_lenagan}).
Note that instead of requiring $\tilde{s}\in S$, we need to include the condition that $as_2=bs_1$ is contained in $S$.
Nevertheless, the two relations share the same equivalence classes.
Indeed, there exist several characterizations of the equivalence relation behind Ore localization, some of which are collected in the forthcoming \Cref{eqrel_characterization}.

\begin{definition}
	Let $S$ be a left Ore set in a domain $R$.
	The \emph{structural homomorphism} or \emph{localization map} of $S\inv R$ is $\rho_{S,R}:R\rightarrow S\inv R,~r\mapsto(1,r)$.
\end{definition}

A short calculation shows that $\rho_{S,R}$ is a monomorphism of rings that satisfies the requirements for $\varphi$ from \Cref{definition_left_ore_localization}.
Therefore, the pair $(S\inv R,\rho_{S,R})$ is indeed a realization of the left Ore localization of $R$ at $S$.

\begin{lemma}
	Let $S$ be a left Ore set in a domain $R$ and $(s,r)\in S\inv R$.
	\begin{enumerate}[(a)]
		\item
			We have $(s,r)=0$ if and only if $r=0$.
		\item
			We have $(s,r)=1$ if and only if $s=r$.
		\item
			Let $w\in R$ with $ws\in S$, then $(s,r)=(ws,wr)$.
		\item
			The left Ore localization $S\inv R$ is a domain.
	\end{enumerate}
\end{lemma}

The following is a generalization of the classical quotient field:

\begin{definition}
	A domain $R$ is called a \emph{left Ore domain} if $R\setminus\set{0}$ is a left Ore set in $R$.
	The associated localization $\Quot(R):=(R\setminus\set{0})\inv R$ is called the \emph{left quotient (skew) field} of $R$.
\end{definition}

As the name might suggest, $\Quot(R)$ is indeed a skew field: the inverse of a non-zero element $(s,r)\in\Quot(R)$ is given by $(r,s)$.
Any left Noetherian domain is a left Ore domain (e.g. \cite{mcconnell_robson}, 2.1.15) and thus has a left quotient field.

\begin{lemma}
	Let $S$ be a left Ore set in a left Ore domain $R$.
	Then $S\inv R$ is a left Ore domain.
\end{lemma}
\begin{proof}
	Let $(s_1,r_1)\in S\inv R$ and $(s_2,r_2)\in S\inv R\setminus\set{0}$, then $r_2\in R\setminus\set{0}$.
	Since $R$ is a left Ore domain there exist $x'\in R\setminus\set{0}$ and $y'\in R$ such that $x'r_1=y'r_2$.
	Define $x:=(1,x')\cdot(1,s_1)\in S\inv R\setminus\set{0}$ and $y:=(1,y')\cdot(1,s_2)\in S\inv R$, then
	\[
		x(s_1,r_1)
		=(1,x')(1,s_1)(s_1,r_1)
		=(1,x'r_1)
		=(1,y'r_2)
		=(1,y')(1,s_2)(s_2,r_2)
		=y(s_2,r_2).\qedhere
	\]
\end{proof}

\section{Saturation closure}

Now we define the notion of \emph{left $T$-closure} or \emph{left $T$-saturation} of $P$, where $T$ is a subset of a ring $R$ and $P$ is a subset of a left $R$-module $M$.

\begin{definition}
	Let $T$ be a subset of a ring $R$, $M$ a left $R$-module and $P$ a subset of $M$.
	Then
	\[
		\LSat_T(P)
		:=\LSat_T^M(P)
		:=\set{m\in M\mid\exists~t\in T:tm\in P}.
	\]
\end{definition}

By construction, $\LSat_T^M(P)$ is inclusion-preserving in all three parameters.
The next results follow immediately from the definition:

\begin{lemma}\label{basic_properties_of_general_LSat}
	Let $T$ be a subset of a ring $R$, $M$ a left $R$-module and $P$ a subset of $M$.
	\begin{enumerate}[(a)]
		\item
			If $1\in T$, then $P\subseteq\LSat_T(P)$.
		\item
			If $T\neq\emptyset$, then $0\in P$ if and only if $0\in\LSat_T(P)$.
		\item
			If $0\in T$, then
			\[
				\LSat_T(P)=M
				\quad\Leftrightarrow\quad
				0\in\LSat_T(P)
				\quad\Leftrightarrow\quad
				0\in P.
			\]
		\item
			If $0\notin P$, then $\LSat_T(P)=\LSat_{T\setminus\set{0}}(P)$.
		\item
			If $S$ is also a subset of $R$, then $\LSat_S(\LSat_T(P))=\LSat_{TS}(P)$.
	\end{enumerate}
\end{lemma}

We restrict ourselves to multiplicative sets $T$ to get more structure.
This has the additional benefit that $P$ is always a subset of $\LSat_T(P)$ since $1\in T$.

\begin{definition}
	Let $T$ be a multiplicative set in a ring $R$, $M$ a left $R$-module and $P$ a subset of $M$.
	We call $P$ \emph{left} $T$-\emph{saturated} if $tm\in P$ implies $m\in P$ for all $t\in T$ and all $m\in M$.
\end{definition}


\begin{lemma}\label{basic_properties of_LSat_for_qmc}
	Let $T$ be a multiplicative set in a ring $R$, $M$ a left $R$-module and $P$ a subset of $M$.
	Then we have:
	\begin{enumerate}[(a)]
		\item
			$\LSat_T(P)$ is left $T$-saturated.
		\item
			$P$ is left $T$-saturated if and only if $P=\LSat_T(P)$. 
		\item
			$\LSat_T(P)$ is the smallest left $T$-saturated superset of $P$ with respect to inclusion.
	\end{enumerate}
\end{lemma}
\begin{proof}
	\begin{enumerate}[{\itshape(a)}]
		\item
			Let $t\in T$ and $m\in M$ such that $tm\in\LSat_T(P)$.
			Then there exists $\tilde{t}\in T$ such that $\tilde{t}tm\in P$.
			Since $\tilde{t}t\in T$, we have $m\in\LSat_T(P)$.
		\item
			Let $m\in\LSat_T(P)$, then $tm\in P$ for some $t\in T$.
			If $P$ is left $T$-saturated, we have $m\in P$, which implies $\LSat_T(P)=P$. 
			On the other hand, since $\LSat_T(P)$ is left $T$-saturated by the result above, $P=\LSat_T(P)$ implies that $P$ is left $T$-saturated.
		\item
			Let $Q\subseteq M$ be a left $T$-saturated set with $P\subseteq Q\subseteq\LSat_T(P)$.
			Let $m\in\LSat_T(P)$, then $tm\in P\subseteq Q$ for some $t\in T$.
			Since $Q$ is left $T$-saturated, we have $m\in Q$ and therefore $\LSat_T(P)=Q$.\qedhere
	\end{enumerate}
\end{proof}

Due to the previous result, we can interpret $\LSat_T(P)$ for a multiplicative set $T$ as the \emph{left} $T$-\emph{saturation closure} of $P$ in $M$.

Consider the following question: given two subsets of $M$, do they have the same left $T$-saturation closure?
A straight-forward approach is to compute both closures and compare, but this might not be feasible.
Nevertheless, there are sufficient conditions for equality of the closures that can be easier to check:

\begin{cor}\label{sufficient_conditions_for_equality_of_closures}
	Let $T$ be a multiplicative set in a ring $R$, $M$ a left $R$-module and $P_1,P_2\subseteq M$ two subsets of $M$.
	\begin{enumerate}[(a)]
		\item
			We have $P_1\subseteq\LSat_T(P_2)$ if and only if $\LSat_T(P_1)\subseteq\LSat_T(P_2)$.
		\item
			The following are equivalent:
			\begin{enumerate}[(1)]
				\item
					$P_1\subseteq\LSat_T(P_2)$ and $P_2\subseteq\LSat_T(P_1)$.
				\item
					$\LSat_T(P_1)=\LSat_T(P_2)$.
			\end{enumerate}
		\item
			If $P_1\subseteq P_2\subseteq\LSat_T(P_1)$, then $\LSat_T(P_1)=\LSat_T(P_2)$.
	\end{enumerate}
\end{cor}
\begin{proof}
	\begin{enumerate}[{\itshape(a)}]
		\item
			Let $P_1\subseteq\LSat_T(P_2)$ and $m\in\LSat_T(P_1)$, then there exists $t\in T$ such that $tm\in P_1\subseteq\LSat_T(P_2)$.
			But then $t'tm\in P_2$ for some $t'\in T$, thus $m\in\LSat_T(P_2)$ since $t't\in T$, which implies $\LSat_T(P_1)\subseteq\LSat_T(P_2)$.
			The other implication is obvious since $P_1\subseteq\LSat_T(P_1)$.
		\item
			Follows from applying \textit{(a)} twice.
		\item
			Follows from \textit{(b)}, since $P_1\subseteq P_2$ implies $P_1\subseteq\LSat_T(P_2)$.\qedhere
	\end{enumerate}
\end{proof}

Note that by \Cref{inclusions_of_multiplicatively_closed_sets_on_generators} it suffices to check the inclusions on generators.

\section{Left saturation with respect to the ring}

\begin{definition}\label{definition_LSat_in_ring}
	Let $P$ be a subset of a domain $R$.
	The \emph{left saturation} of $P$ in $R$ is
	\[
		\LSat(P)
		:=\LSat_R^R(P)
		=\set{r\in R\mid\exists~w\in R:wr\in P}.
	\]
\end{definition}

Strictly speaking, $R$ is not a multiplicative set since it obviously contains $0$.
Nevertheless, in light of \Cref{basic_properties_of_general_LSat} we have $\LSat(P)=\LSat_R(P)=\LSat_{R\setminus\set{0}}(P)$ since $0\notin M$.
Now $R\setminus\set{0}$ is a multiplicative set since $R$ is a domain and we can apply \Cref{basic_properties of_LSat_for_qmc}:

\begin{lemma}\label{basic_properties_of_LSat_for_N=T=R}
	Let $P$ be a subset of a domain $R$.
	Then we have:
	\begin{enumerate}[(a)]
		\item
			$\LSat(P)$ is left saturated.
		\item
			$P$ is left saturated if and only if $P=\LSat(P)$.
		\item
			$\LSat(P)$ is the smallest left saturated superset of $P$ with respect to inclusion.
		
	\end{enumerate}
\end{lemma}

Furthermore, we have the following connection to the units of the base ring:

\begin{lemma}\label{units_are_contained_in_left_saturation_closure}
	Let $R$ be a domain.
	\begin{enumerate}[(a)]
		\item
			For any non-empty subset $P$ of $R$ we have $U(R)\subseteq\LSat(P)$.
		\item
			We have $\LSat(\set{1})=U(R)$.
			Furthermore, for any $\set{1}\subseteq U\subseteq U(R)$, we have $\LSat(U)=U(R)$.
	\end{enumerate}
\end{lemma}
\begin{proof}
	\begin{enumerate}[{\itshape(a)}]
		\item
			Let $u\in U(R)$ and $p\in P$.
			Then $p\cdot u\inv\cdot u=p\in P$ and $p\cdot u\inv\in R$, thus $u\in\LSat(P)$.
		\item
			Let $x\in\LSat(\set{1})$, then there exists $w\in R\setminus\set{0}$ such that $wx=1$.
			Then $x\in U(R)$ since $R$ is a domain and thus Dedekind-finite.
			Together with \textit{(a)} we have $\LSat(\set{1})=U(R)$.\\
			Additionally, $\set{1}\subseteq U\subseteq U(R)=\LSat(\set{1})$ implies $\LSat(U)=U(R)$ by \Cref{sufficient_conditions_for_equality_of_closures}.\qedhere
	\end{enumerate}
\end{proof}

\begin{ex}\label{LSat_of_Theta}
	In $\mc{D}$ (over the field $K$), we have
	\[
		\LSat(\Theta)
		=\merz{\Theta\cup\set{x,\partial}\cup U(K)}
		=\merz{(\theta+\IZ)\cup\set{x,\partial}\cup(K\setminus\set{0})}.
	\]
	To see this, let $S:=\merz{(\theta+\IZ)\cup\set{x,\partial}\cup U(K)}$.
	Since $\theta=x\partial\in\Theta$, we clearly have $\set{x,\partial}\subseteq\LSat(\Theta)$ and therefore $S\subseteq\LSat(\Theta)$ (note that $U(K)=K\setminus\set{0}=U(\mc{D})$ is always contained in $\LSat(\Theta)$ by \Cref{units_are_contained_in_left_saturation_closure}).\\
	To see the other inclusion, consider that by \Cref{theta_proposition}, every element $s\in S$ can be written in the form $s=ty^n$, where $y\in\set{x,\partial}$, $n\in\IN_0$ and $t\in\Theta$.
	Then \cite{homogfac} implies that every other non-trivial factorization of $s$ can be derived by using the commutation rules given in \Cref{theta_proposition} and rewriting $\theta$ respectively $\theta+1$ as $x\partial$ respectively $\partial x$.
	But all occurring factors are already contained in $S$, thus $\LSat(\Theta)\subseteq S$ (the trivial factorizations correspond to scattering units between the factors).
\end{ex}

\begin{rem}
	In the following, we omit units when giving generators of saturation closures for brevity, writing e.g. $\LSat(\Theta)=\merz{(\theta+\IZ)\cup\set{x,\partial}}$ in the situation of \Cref{LSat_of_Theta}.
\end{rem}

Now we consider left saturation of left Ore sets.
As a first application, this gives us a complete characterization of the units in the associated localization:

\begin{prop}\label{characterization_of_units_in_localization}
	Let $S\subseteq R$ be an Ore set in a domain $R$ and $(s,r)\in S\inv R$.
	The following are equivalent:
	\begin{enumerate}[(1)]
		\item
			$(s,r)\in U(S\inv R)$.
		\item
			$(1,r)\in U(S\inv R)$.
		\item
			$r\in\LSat(S)$.
	\end{enumerate}
\end{prop}
\begin{proof}
	We always have the basic factorization $(s,r)=(s,1)\cdot(1,r)$.
	The equivalence of \textit{(1)} and \textit{(2)} is then due to the fact that $(s,1)$ is a unit in $S\inv R$ with inverse $(1,s)$.
	Now we prove the equivalence of \textit{(2)} and \textit{(3)}:\\
	First, let $(1,r)\in U(S\inv R)$.
	Then there exists $(s,w)\in S\inv R$ such that $(1,1)=(s,w)\cdot(1,r)=(s,wr)$, which implies $wr=s\in S$ and thus $r\in\LSat(S)$.\\
	On the other hand, let $r\in\LSat(S)$ with $w\in R$ such that $wr\in S$.
	Then $(wr,w)\in S\inv R$ satisfies $(wr,w)\cdot(1,r)=(wr,wr)=(1,1)$ and thus $(1,r)\in U(S\inv R)$.
\end{proof}

\section{Homomorphisms of localized rings}

Before we continue investigating the connection of $S$ and $\LSat(S)$ further in \Cref{section_localization_at_left_saturation}, we characterize homomorphisms between localizations of the same ring $R$ that leave $R$ unchanged:

\begin{definition}
	Let $S$ and $T$ be left Ore sets in a domain $R$.
	A homomorphism $\varphi:S\inv R\rightarrow T\inv R$ \emph{fixes} $R$ or is an $R$-\emph{fixing homomorphism} if $\varphi\circ\rho_{S,R}=\rho_{T,R}$.
\end{definition}

If $\varphi$ is an $R$-fixing isomorphism, then $\varphi\inv$ fixes $R$ as well.

\begin{restatable}{lemma}{omegalemma}\label{omega_lemma}
	Let $S$ and $T$ be left Ore sets in a domain $R$ such that $S\subseteq\LSat(T)$.
	For any $s\in S$ choose a $w_s\in R$ such that $w_ss\in T$.
	Then the map
	\[
		\omega:=\omega_{S,T,R}:S\inv R\rightarrow T\inv R,~(s,r)\mapsto(w_ss,w_sr),
	\]
	\begin{enumerate}[(a)]
		\item
			is independent of the choice of the $w_s$ and
		\item
			an injective $R$-fixing ring homomorphism.
	\end{enumerate}
\end{restatable}
\begin{proof}
	Due to its rather technical nature, the proof can be found in \Cref{appendix_omega_lemma}.
\end{proof}

\begin{thm}\label{characterization_of_R-fixing_homomorphisms_via_LSat}
	Let $S$ and $T$ be left Ore sets in a domain $R$.
	Then the following are equivalent:
	\begin{enumerate}[(1)]
		\item
			$S\subseteq\LSat(T)$.
		\item
			$\omega$ is the unique $R$-fixing ring homomorphism from $S\inv R$ to $T\inv R$.
	\end{enumerate}
\end{thm}
\begin{proof}
	\begin{description}
		\implication{1}{2}
			By \Cref{omega_lemma} the map $\omega$ is an $R$-fixing ring homomorphism.
			Consider another $R$-fixing ring homomorphism $\varphi:S\inv R\rightarrow T\inv R$ and $(s,r)\in S\inv R$, then there exists $w\in R$ such that $ws\in T$.
			Now $(ws,w)\in T\inv R$ is the inverse of $(1,s)=\rho_{T,R}(s)$ in $T\inv R$, thus
			\[\begin{split}
				\varphi((s,r))
				&=\varphi((s,1)\cdot(1,r))
				=\varphi((s,1))\cdot\varphi((1,r))
				=\varphi((s,1))\cdot(1,r)\\
				&=\varphi((1,s)\inv)\cdot(1,r)
				=\varphi((1,s))\inv\cdot(1,r)
				=(1,s)\inv\cdot(1,r)\\
				&=(ws,w)\cdot(1,r)
				=(ws,wr)
				=\omega((s,r)).
			\end{split}\]
		\implication{2}{1}
			Let $s\in S$, then there exist $t\in T$ and $r\in R$ such that $\omega((s,1))=(t,r)$.
			Now
			\[
				1_{T\inv R}
				=\omega(1_{S\inv R})
				=\omega((s,s))
				=\omega((s,1))\cdot\omega((1,s))
				=(t,r)\cdot(1,s)
				=(t,rs),
			\]
			since $\omega$ fixes $R$.
			Thus $rs=t\in T$ and $s\in\LSat(T)$, which implies $S\subseteq\LSat(T)$.\qedhere
	\end{description}
\end{proof}

\begin{rem}
	In the light of the previous theorem it is justified to call $\omega$ the \emph{canonical} $R$-fixing homomorphism from $S\inv R$ to $T\inv R$.
	If $S\subseteq T$, then $\omega((s,r))=(s,r)$.
	The structural homomorphism $\rho_{T,R}$ is the special case where $S=\set{1}$.
\end{rem}

\begin{cor}\label{characterization_of_R-fixing_isomorphisms_via_LSat}
	Let $S$ and $T$ be left Ore sets in a domain $R$.
	Then the following are equivalent:
	\begin{enumerate}[(1)]
		\item
			$\LSat(S)=\LSat(T)$.
		\item
			$\omega$ is the unique $R$-fixing ring isomorphism from $S\inv R$ to $T\inv R$.
	\end{enumerate}
\end{cor}
\begin{proof}
	\begin{description}
		\implication{1}{2}
			The only thing left to show is surjectivity of $\omega$, the rest follows from \Cref{characterization_of_R-fixing_homomorphisms_via_LSat}: let $(t,r)\in T\inv R$, then there exists $w\in R$ such that $wt\in S$ since $T\subseteq\LSat(S)$.
			Furthermore there exists $q\in R$ such that $qwt\in T$ since $S\subseteq\LSat(T)$.
			Now we have
			\[
				(t,r)
				=(qwt,qwr)
				=\omega((wt,wr))
				\in\im(\omega),
			\]
			thus $\omega$ is surjective.
		\implication{2}{1}
			Applying \Cref{characterization_of_R-fixing_homomorphisms_via_LSat} to $\omega$ we get $S\subseteq\LSat(T)$, applying the same result to $\omega\inv$ gives us $T\subseteq\LSat(S)$, which implies $\LSat(S)=\LSat(T)$ by \Cref{sufficient_conditions_for_equality_of_closures}.\qedhere
	\end{description}
\end{proof}

\section{Localization at left saturation}\label{section_localization_at_left_saturation}

Given an arbitrary subset $P$ of a domain $R$, $\LSat(P)$ is not (right-)saturated in general:

\begin{ex}
	The set $T:=\LSat(\set{x\partial+1})$ is not (right-)saturated in $\mc{D}$.
	To see this, assume that $T$ is saturated, then $\partial\in T$ since $\partial x=x\partial+1$.
	Since $T$ is the left saturation of $x\partial+1$ there exists a $w\in\mc{D}$ such that $w\partial=x\partial+1$, or equivalently, $(w-x)\partial=1$.
	This implies that $\partial$ is a unit in $\mc{D}$, which is a contradiction.
\end{ex}

In the case where $S$ is a left Ore set, we can show that $\LSat(S)$ is indeed saturated via a little trick that involves the localization $S\inv R$:

\begin{prop}\label{LSat_of_Ore_set_is_saturated}
	Let $S\subseteq R$ be a left Ore set in a domain $R$.
	Then $\LSat(S)$ is saturated.
\end{prop}
\begin{proof}
	Let $p,q\in R$ such that $r:=pq\in\LSat(S)$.
	By \Cref{characterization_of_units_in_localization}, we have that
	\[
		(1,p)\cdot(1,q)
		=(1,pq)
		=(1,r)
		\in U(S\inv R).
	\]
	Since $U(S\inv R)$ is saturated by \Cref{unit_group_is_saturated}, we have $(1,p),(1,q)\in U(S\inv R)$.
	But then $p,q\in\LSat(S)$ by \Cref{characterization_of_units_in_localization}.
\end{proof}

The fact that $\LSat(S)$ is saturated on both sides gives us further insight into the defining equivalence relation of Ore localization and its variants:

\begin{lemma}\label{eqrel_characterization}
	Let $S$ be a left Ore set in a domain $R$ and $(s_1,r_1),(s_2,r_2)\in S\times R$.
	The following are equivalent:
	\begin{enumerate}[(1)]
		\item
			$(s_1,r_1)\sim(s_2,r_2)$, that is, there exist $\tilde{s}\in S$ and $\tilde{r}\in R$ such that $\tilde{s}s_2=\tilde{r}s_1$ and $\tilde{s}r_2=\tilde{r}r_1$.
		\item
			For all $\hat{s}\in S$ and $\hat{r}\in R$ such that $\hat{s}s_2=\hat{r}s_1$ there exists $\bar{s}\in S$ such that $\bar{s}\hat{s}r_2=\bar{s}\hat{r}r_1$.
		\item
			There exist $\hat{s},\bar{s}\in S$ and $\hat{r}\in R$ such that $\hat{s}s_2=\hat{r}s_1$ and $\bar{s}\hat{s}r_2=\bar{s}\hat{r}r_1$.
		\item
			$(s_1,r_1)\approx(s_2,r_2)$, that is, there exist $a,b\in R$ such that $as_2=bs_1\in S$ and $ar_2=br_1$.
		\item
			There exist $\mathring{s}\in S$ and $\mathring{x}\in\LSat(S)$ such that $\mathring{s}s_2=\mathring{x}s_1$ and $\mathring{s}r_2=\mathring{x}r_1$.
	\end{enumerate}
\end{lemma}
\begin{proof}
	The first four equivalence relations are more or less well-known ways to construct Ore localizations.
	Since \textit{(5)} is a special case of \textit{(1)}, the only thing to show is that \textit{(1)} implies \textit{(5)}: from $\mathring{s}:=\tilde{s}\in S$ we know that $\tilde{r}s_1=\tilde{s}s_2\in S$ and therefore $\mathring{x}:=\tilde{r}\in\LSat(S)$.
\end{proof}

In particular, if $S$ is already saturated, then for any pair $s_1,s_2\in S$ we can $\tilde{s}_1,\tilde{s}_2\in S$ realizing the left Ore condition for $s_1$ and $s_2$.

Left saturation preserves and reflects the left Ore condition:

\begin{lemma}\label{LSat_preserves_and_reflects_left_Ore_condition}
	Let $S$ be a subset of a domain $R$.
	The following are equivalent:
	\begin{enumerate}[(1)]
		\item
			$S$ satisfies the left Ore condition in $R$.
		\item
			$\LSat(S)$ satisfies the left Ore condition in $R$.
	\end{enumerate}
\end{lemma}
\begin{proof}
	\begin{description}
		\item[\normalfont\textit{(1)}$\Rightarrow$\textit{(2)}:]
			Let $r\in R$ and $x\in\LSat(S)$, then by definition $wx\in S$ for some $w\in R$.
			By the left Ore condition on $S$, there exist $\tilde{r}\in R$ and $\tilde{s}\in S$ such that $\tilde{r}wx=\tilde{s}r$.
			Since $\tilde{r}w\in R$ and $\tilde{s}\in S\subseteq\LSat(S)$, we have that $\LSat(S)$ satisfies the left Ore condition in $R$.
		\item[\normalfont\textit{(2)}$\Rightarrow$\textit{(1)}:]
			Let $r\in R$ and $s\in S\subseteq\LSat(S)$.
			By the left Ore condition on $\LSat(S)$, there exist $\tilde{r}\in R$ and $x\in\LSat(S)$ such that $\tilde{r}s=xr$.
			By definition, $wx\in S$ for some $w\in R$, thus $w\tilde{r}s=wxr$.
			Since $wx\in S$ and $w\tilde{r}\in R$, we have that $S$ satisfies the left Ore condition in $R$.\qedhere
	\end{description}
\end{proof}

Unfortunately, left saturation does not preserve multiplicative closedness in general:

\begin{ex}
	Consider $S=\merz{\theta}=\merz{x\partial}$ in $\mc{D}$.
	Clearly, $\partial\in\LSat(S)$, but $\partial^2\notin\LSat(S)$: assume that there exist $w\in\mc{D}$ and $k\in\IN_0$ such that $w\partial^2=\theta^k=(x\partial)^k$.
	Then $k\geq2$ and $w\partial=(x\partial)^{k-1}x=x(\partial x)^{k-1}=x(x\partial+1)^{k-1}=f\partial+x$ for some $f\in\mc{D}$.
	This implies that $x\in\mc{D}\partial$, which is a contradiction.
\end{ex}

Fortunately, the left Ore condition is sufficient to overcome this obstacle:

\begin{prop}
	Let $S$ be a left Ore set in a domain $R$.
	Then $\LSat(S)$ is a left Ore set in $R$ and $S\inv R\cong\LSat(S)\inv R$.
\end{prop}
\begin{proof}
	By \Cref{LSat_preserves_and_reflects_left_Ore_condition}, $\LSat(S)$ satisfies the left Ore condition, and by \Cref{basic_properties_of_LSat_for_N=T=R}, $0\notin\LSat(S)$ since $0\notin S$.
	Furthermore, we have $1\in S\subseteq\LSat(S)$.
	Finally, let $x,y\in\LSat(S)$, then there exist $a,b\in R\setminus\set{0}$ such that $ax\in S$ and $by\in S$.
	By the Ore condition on $S$, there exist $\tilde{s}\in S$ and $\tilde{r}\in R$ such that $\tilde{r}ax=\tilde{s}b$.
	Then we have $\tilde{r}axy=\tilde{s}by\in S$ and therefore $xy\in\LSat(S)$, thus $\LSat(S)$ is a multiplicative set.
	Thus, $\LSat(S)$ is a left Ore set in $R$.\\
	Lastly, $S\inv R\cong\LSat(S)\inv R$ follows from $\LSat(S)=\LSat(\LSat(S))$ with \Cref{characterization_of_R-fixing_isomorphisms_via_LSat}.
\end{proof}

At this point we can see that, for every left Ore set $S$, $\LSat(S)$ gives us a saturated left Ore set that describes the same localization up to a unique $R$-fixing isomorphisms.
For theoretical purposes, this allows us to assume without loss of generality that any given left Ore set is already saturated.

The main results of the theory up to this point can thus be summarized as follows:

\begin{thm} 
	Let $S$ be a left Ore set in a domain $R$.
	\begin{enumerate}[(a)]
		\item
			$\LSat(S):=\set{r\in R\mid\exists w\in R:wr\in S}$ is a left Ore set in $R$.
		\item
			$\LSat(S)=\rho_{S,R}\inv(\units{S\inv R})$, or equivalently, $(s,r)\in\units{S\inv R}$ if and only if $r\in\LSat(S)$.
		\item
			$\LSat(S)$ is the smallest saturated superset of $S$ with respect to inclusion.
		\item
			$S\inv R\cong_R\LSat(S)\inv R$ via a unique $R$-fixing isomorphism.
	\end{enumerate}
\end{thm}

\section{Examples}

First, we give a detailed proof of the fact that $\Theta_z=\merz{\theta+z+\IZ/p\IZ}$ is a left Ore set in $\mc{D}$ over the field $K$ with $p:=\chara(K)\in\IP\cup\set{0}$ for any $z\in Z(\mc{D})$, the center of $\mc{D}$.
Note that $\Theta_z=\Theta_{z+w}$ for all $z\in Z(\mc{D})$ and $w\in\IZ/p\IZ$.
Here, $\IZ/p\IZ=\IZ$ for characteristic zero, and for $p>0$ we identify $\IZ/p\IZ$ with the prime field of $K$.

\begin{prop}\label{theta_proposition}
	Let $z\in Z(\mc{D})$.
	\begin{enumerate}[(a)]
		\item
			For all $m,n\in\IN_0$ we have $(\theta+z)^mx^n=x^n(\theta+z+n)^m$ and $\partial^n(\theta+z)^m=(\theta+z+n)^m\partial^n$.
		\item
			Let $r\in\mc{D}$ be homogeneous of degree $k\in\IZ$.
			Then $(\theta+z+k)r=r(\theta+z)$.
		\item
			For all $r\in\mc{D}$ there exist $\tilde{s}\in\Theta_z$ and $\tilde{r}\in\mc{D}$ such that $\tilde{r}(\theta+z)=\tilde{s}r$.
	\end{enumerate}
\end{prop}
\begin{proof}
	\begin{enumerate}[(a)]
		\item
			The statement follows by induction on $n$ and $m$ from
			\[
				(\theta+z)x
				=\theta x+zx
				=x\partial x+zx
				=x(x\partial+1)+xz
				=x(x\partial+z+1)
				=x(\theta+z+1)
			\]
			and $\partial(\theta+z)=(\theta+z+1)\partial$.
		\item
			Since $r$ is homogeneous of degree $k$, we have a representation $r=\sum_{\substack{(a,b)\in\IN_0^2\\b-a=k}}^{}c_{a,b}x^a\partial^b$.
			Then
			\[\begin{split}
				(\theta+z+k)r
				&=\sum_{\substack{(a,b)\in\IN_0^2\\b-a=k}}^{}c_{a,b}(\theta+z+k)x^a\partial^b
				=\sum_{\substack{(a,b)\in\IN_0^2\\b-a=k}}^{}c_{a,b}x^a\partial^b(\theta+z)
				=r(\theta+z).
			\end{split}\]
		\item
			Let $r=\sum_{i=1}^{n}r_{k_i}$ be the decomposition of $r$ into its homogeneous parts, where $r_{k_i}\in\mc{D}$ has degree $k_i\in\IZ$.
			A short calculation shows that the elements
			\[
				\tilde{s}
				:=\prod_{i=1}^{n}(\theta+z+k_i)
				\in\Theta_z
				\quad\text{and}\quad
				\tilde{r}
				:=\sum_{i=1}^{n}\Bigg(\prod_{\substack{j=1\\j\neq i}}^{n}(\theta+z+k_j)\Bigg)r_{k_i}
				\in\mc{D}.
			\]
			satisfy the equation $\tilde{r}(\theta+z)=\tilde{s}r$.\qedhere
	\end{enumerate}
\end{proof}

\begin{prop}\label{Theta_is_left_Ore}
	\begin{enumerate}[(a)]
		\item
			The set $\Theta_z=\merz{\theta+z+\IZ/p\IZ}$ is a left Ore set in $\mc{D}$ for all $z\in Z(\mc{D})$.
		\item
			We have $\LSat(\Theta_0)=\LSat(V)$, where $V=\merz{x,\partial}$.
		\item
		For $z\in K$ such that $z\notin \IZ/p\IZ$, we have $\LSat(\Theta_z)=\Theta_z$.
	\end{enumerate}
	
\end{prop}
\begin{proof}
	\begin{enumerate}[(a)]
		\item
			Let $r\in\mc{D}$, $w_1,w_2\in\IZ/p\IZ$ and consider $s:=(\theta+z+w_1)(\theta+z+w_2)\in\Theta_z$.
			By \Cref{theta_proposition}, there exist $s_2\in\Theta_{z+w_2}=\Theta_z$ and $r_2\in\mc{D}$ such that $s_2r=r_2(\theta+z+w_2)$.
			Again by \Cref{theta_proposition}, there exist $s_1\in\Theta_{z+w_1}=\Theta_z$ and $r_1\in\mc{D}$ such that $s_1r_2=r_1(\theta+z+w_1)$.
			Define $\tilde{s}:=s_1s_2\in\Theta$ and $\tilde{r}:=r_1\in\mc{D}$, then
			\[
				\tilde{r}s
				=r_1(\theta+z+w_1)(\theta+z+w_2)
				=s_1r_2(\theta+z+w_2)
				=s_1s_2r
				=\tilde{s}r.
			\]
			Since every element of $\Theta_z$ has the form $\prod_{i=1}^{n}(\theta+z+w_i)$ for some $n\in\IN_0$ and $w_i\in\IZ/p\IZ$, the statement follows by induction on $n$.
		\item
			By \Cref{sufficient_conditions_for_equality_of_closures} and \Cref{inclusions_of_multiplicatively_closed_sets_on_generators} it suffices to show the inclusions $\theta+\IZ/p\IZ\subseteq\LSat(V)$ and $\set{x,\partial}\subseteq\LSat(\Theta_0)$.
			The latter is obvious from $x\cdot\partial=\theta\in\Theta_0$ and $\partial\cdot x=\theta+1\in\Theta_0$, to see the former let $w\in\IZ/p\IZ$.
			First consider the case $p=0$:
			\begin{description}
				\item[Case $1$:]
					If $w\in\IN_0$, then $x^w(\theta+w)=\theta x^w=x\partial x^w\in V$.
				\item[Case $2$:]
					If $-w\in\IN$, then $\partial^{-w}(\theta+w)=(\theta+w-w)\partial^{-w}=\theta\partial^{-w}=x\partial^{1-w}\in V$.
			\end{description}
			If $p>0$ we can always find $n\in\IN_0$ such that $w=n+p\IZ$ and we get $x^n(\theta+w)=\theta x^n\in V$.
			In each case there exists a $s\in V$ such that $s(\theta+w)\in V$, thus $\theta+w\in\LSat(V)$, which shows $\theta+\IZ/p\IZ\subseteq\LSat_V(V)\subseteq\LSat(V)$.
		\item
			By Lemma 1.12  from \cite{homogfac}, a $\IZ$-graded polynomial factorizes into $\IZ$-graded factors only.
			However, a polynomial which is irreducible in $K[\theta]$ can be reducible in $\mc{D}$.
			Lemma 2.3 from \cite{homogfac} assures us that only $\theta$ and $\theta+1$ are such polynomials.
			
			Now fix $z\notin\IZ/p\IZ$ as in the statement and let $f=\prod_{i}^{}(\theta+z-k_i)\in\Theta_z$ for some $k_i\in\IZ/p\IZ$.
			Since $K[\theta]$ is a UFD, neither $\theta$ nor $\theta+1$ appears as a factor.
			By \cite{homogfac}, no other factorization of $f$ in $\mc{D}$ exists up to permutation.
			
			Note, that for the case $p>0$, $\prod_{k=0}^{p-1}(\theta+z+k) = \theta^p - \theta + z^p - z = x^p \d^p + z^p - z$ is contained in the center $Z(\mc{D})$ and, since $z\notin \IZ/p\IZ$, also $z^p-z \neq 0$.
			Therefore the localization at $\Theta_z$ is isomorphic to a central localization.\qedhere
	\end{enumerate}
\end{proof}

\begin{ex}
	Consider the left Ore set $S:=\merz{\Theta_0\cup\set{\partial-1}}=\merz{(\theta+\IZ)\cup\set{\partial-1}}$ in characteristic zero.
	Makar-Limanov showed in \cite{makar-limanov} that the skew field of fractions of the first Weyl algebra, which is the localization $(\mc{D}\setminus\set{0})\inv\mc{D}$, contains a free subalgebra generated by the elements $(\partial x,1)$ and $(\partial x,1)\cdot(1-\partial,1)$.
	These two elements can also be found in the (smaller) localization $S\inv\mc{D}$.\\
	In contrast to $\LSat(\Theta_0)$, $\LSat(S)$ is inhomogeneous and thus much harder to describe, for example, for all $i\in\IZ$ we have
	\[
		(x\partial+i+1)(x\partial^2-x\partial+(i+2)\partial-i)
		=(\partial-1)(x\partial+i)(x\partial+i+1)\in S,
	\]
	thus $\LSat(S)$ contains $x\partial^2-x\partial+(i+2)\partial-i$, which is irreducible for almost all $i$.
\end{ex}

Finally, we give an example in the commutative setting:

\begin{prop}
	Let $K$ be a field and consider the Laurent polynomial ring $R=K[x,x\inv]$ as well as $S=K[x]\setminus\set{0}$.
	\begin{enumerate}[(a)]
		\item
			$S$ is a multiplicative set in $R$, but not saturated.
		\item
			$\LSat(S)=R\setminus\set{0}$.
	\end{enumerate}
\end{prop}
\begin{proof}
	\begin{enumerate}[{\itshape(a)}]
		\item
			$K[x]$ is a subring of $R$ and a domain.
			We have $x\cdot x\inv=1\in S$, but $x\inv\notin S$.
		\item
			Let $r\in R\setminus\set{0}$, then there is a representation $r=x^{-k}f$, where $f\in S$ and $k\in\IN_0$.
			Now $x^kr=x^kx^{-k}f=f\in S$ implies $r\in\LSat(S)$.
			Since $x^k\in U(R)$, we also have the stronger property $\LSat_{U(R)}(S)=R\setminus\set{0}$.\qedhere
	\end{enumerate}
\end{proof}

In non-commutative algebra, the \emph{Gel'fand-Kirillov dimension} ($\GKdim$) plays the central role in investigations of rings and modules.
See \cite{krause_lenagan,mcconnell_robson} for the definition and studies of its properties.

The following is a very important result due to Borho and Kraft.
See Theorem 4.4.12 in \cite{krause_lenagan} for the general statement, here we give the version for domains:

\begin{thm}[Borho and Kraft (1976)]\label{BorhoKraft}
Let $A$ be a $K$-algebra, which is a finitely generated domain of finite Gel'fand-Kirillov dimension.
Let $B\subseteq A$ be a $K$-subalgebra such that $\GKdim(A)<\GKdim(B)+1$.
Then $S:=B\setminus\set{0}$ is a left (and right) Ore set in $A$ and $S\inv A\cong_A\Quot(A)$.
Moreover, $\Quot(A)$ is a finite dimensional (left or right) vector space over the division ring $S\inv B=\Quot(B)$.
\end{thm}

\begin{rem}
Note that the other direction in \Cref{BorhoKraft}, when we drop the assumption of the finite generation of $A$, is wrong in general: let $K$ be of characteristic zero, $B=K\erz{x,\theta\mid\theta x=x(\theta+1)}\subseteq\mc{D}$ the first polynomial shift algebra (which is a Noetherian domain), $T:=K[\theta]\setminus\set{0}$, $A:=T\inv B\cong K(\theta)\erz{x\mid xq(\theta)=q(\theta-1)x\text{ for all }q\in K(\theta)}$.
It is known (by e.~g. \cite{krause_lenagan}) that $\GKdim(B)=2$ and $\GKdim(A)=3$, so the main condition of \Cref{BorhoKraft} is not satisfied.
Nevertheless we have that $S:=B\setminus\set{0}\subseteq A$ is a left (and right) Ore set and $S^{-1}A\cong_A\Quot(A)\cong\Quot(B)$ holds.
Note that $\GKdim(\Quot(B))=\GKdim(\Quot(\mc{D}))=\infty$.

Moreover, for arbitrary natural $n\geq 1$ there are $K$-algebras $B_n\subseteq A_n$ with
$\GKdim(A) = \GKdim(B) + n$, such that $B\setminus\set{0}$ is an Ore set and 
$(B\setminus\set{0})^{-1} A \cong_A\Quot(A)$. 
Examples of such $B_n$ and $A_n$ are, for instance, multivariate versions of the mentioned algebras which are of Gel'fand-Kirillov dimension $2n$ and $3n$ respectively.
\end{rem}

\section{Towards classifying Ore localizations in factorization domains}

Later, we will make use of the following shorthand:

\begin{definition}
	Let $S$ be a left Ore set in a domain $R$.
	We call $S\inv R$ a \emph{saturated} localization if $S$ is saturated.
\end{definition}

\begin{definition}
	Let $R$ be a domain and $r\in R\setminus(\set{0}\cup U(R))$.
	Then $r$ is called
	\begin{itemize}
		\item
			\emph{reducible} if $r$ can be written as the product of two non-units in $R$, that is, there exist $p,q\in R\setminus U(R)$ such that $r=pq$,
		\item
			\emph{irreducible} if $r$ is not reducible.
	\end{itemize}
\end{definition}

\begin{prop}
	Let $S\subseteq R$ be a left Ore set in a domain $R$, $p,q\in R\setminus\set{0}$ and $r=pq$.
	\begin{enumerate}[(a)]
		\item
			We have $\abs{\set{p,q}\cap\LSat(S)}=2$ if and only if $(1,r)\in U(S\inv R)$.
		\item
			If $\abs{\set{p,q}\cap\LSat(S)}=0$, then $(1,r)$ is reducible in $S\inv R$ and $r$ is reducible in $R$.
		\item
			If $(1,r)$ is irreducible in $S\inv R$, then $\abs{\set{p,q}\cap\LSat(S)}=1$.
	\end{enumerate}
\end{prop}
\begin{proof}
	We always have the induced factorization $(1,r)=(1,p)\cdot(1,q)$ in $S\inv R$.
	\begin{enumerate}[(a)]
		\item
			By \Cref{characterization_of_units_in_localization}, $\abs{\set{p,q}\cap\LSat(S)}=2$ is equivalent to $(1,p),(1,q)\in U(S\inv R)$.
			Since $S\inv R$ is a domain this is equivalent to $(1,r)\in U(S\inv R)$ by \Cref{unit_group_is_saturated}.
		\item
			Since $U(R)\subseteq\LSat(S)$ by \Cref{units_are_contained_in_left_saturation_closure}, we have that $r=pq$ is the product of two non-zero non-units of $R$, thus $r$ is reducible in $R$.
			Furthermore we get that both $(1,p)$ and $(1,q)$ are non-zero non-units in $S\inv R$ by \Cref{characterization_of_units_in_localization}, thus their product $(1,r)$ is reducible in $S\inv R$.
		\item
			Follows from combining the previous parts.\qedhere
	\end{enumerate}
\end{proof}

\begin{definition}
	If every non-zero non-unit of a domain $R$ can be written as a product of finitely many irreducible elements we call $R$ a \emph{factorization domain}.
	Two elements $a$ and $b$ in a factorization domain $R$ are \emph{associated} if there exists $u\in\units{R}$ such that $a=ub$.
\end{definition}

\begin{rem}
	All $G$-algebras like the first Weyl algebra $\mc{D}$ are factorization domains.
	In fact, they are even \emph{finite factorization domains}: any non-zero non-unit of a $G$-algebra has at least one but at most finitely many factorizations into irreducible elements (\cite{BHL-FFD}).
\end{rem}

\begin{lemma}
	Let $S$ be a left Ore set in a factorization domain $R$.
	Up to units $\operatorname{LSat}(S)$ is generated by a set of irreducible elements of $R$ that are pairwise not associated.
\end{lemma}
\begin{proof}
	Start with $T:=\operatorname{LSat}(S)$ as a trivial generating set of itself.
	Let $s\in T$ be reducible.
	Since $R$ is a factorization domain $s$ can be factorized into finitely many irreducible elements $r_1,\ldots,r_k$.
	But $\operatorname{LSat}(S)$ is saturated, thus all $r_i$ already belong to $\operatorname{LSat}(S)$ and we can omit $s$ from $T$.
	Iteratively we can remove all reducible elements from $T$ and are left with only irreducible elements and units.
	Similarly, given two associated irreducible elements in $\operatorname{LSat}(S)$ one is a unit multiple of the other.
	Since all units are contained in $T$ we can remove one of the two irreducible elements from $T$ and still retain a generating set.
\end{proof}

\begin{prop}
	Let $R$ be a factorization domain.
	Let $\mathcal{M}$ be a set of representatives of the equivalence classes of irreducible elements in $R$ with respect to associativity.
	Let $\mathcal{S}$ be a set of representatives of the equivalence classes of saturated left Ore sets in $R$ with respect to natural isomorphisms of the induced localizations.
	Then $\varphi:\mathcal{S}\rightarrow\operatorname{Pot}(\mathcal{M}),~S\mapsto S\cap\mathcal{M}$ is an injective map.
\end{prop}
\begin{proof}
	Any saturated left Ore set in $R$ is generated by $(S\cap\mathcal{M})\cup U(R)$, thus $S$ can be uniquely reconstructed from $S\cap\mathcal{M}$.
\end{proof}

\begin{rem}
	This induces a bijection between $\mathcal{S}$ and the subsets of $\mathcal{M}$ that generate a saturated left Ore set. 
\end{rem}

\begin{cor}
	Let $R$ be a commutative UFD.
	Let $\mathcal{M}$ be a set of representatives of the equivalence classes of irreducible elements in $R$ with respect to associativity.
	Let $\mathcal{S}$ be a set of representatives of the equivalence classes of saturated multiplicative sets in $R$ with respect to natural isomorphisms of the induced localizations.
	Then there is a bijection $\varphi:\mathcal{S}\rightarrow\operatorname{Pot}(\mathcal{M}),~S\mapsto S\cap\mathcal{M}$.
\end{cor}
\begin{proof}
	Let $M$ be a subset of $\mathcal{M}$, then $S:=\merz{M\cup U(R)}$ is a multiplicative set in $R$ that satisfies $M=S\cap\mathcal{M}$.
	Since any element of $S$ up to associativity factors uniquely only into irreducible elements contained in $M$ we have that $S$ is saturated and thus $S\in\mathcal{S}$.
\end{proof}

Before proceeding with further examples, let us give a definition.

\begin{definition}\label{MaxPreMax}
	Let $S$ be a left Ore set in a left Ore domain $R$.
	We call $S$ 
	\begin{itemize} 
		\item
			\emph{maximal} if $S^{-1}R\cong_R\Quot(R)$,
		\item
			\emph{pre-maximal} if $S$ is not maximal, but any left Ore set $T$ with $\LSat(S)\subsetneq T$ is maximal.
	\end{itemize} 
\end{definition}

From \Cref{characterization_of_R-fixing_isomorphisms_via_LSat} we get an immediate characterization of maximality:

\begin{prop}\label{characterization_of_maximal_Ore_sets}
	Let $S$ be a left Ore set in a left Ore domain $R$.
	Then the following are equivalent:
	\begin{enumerate}[(1)]
		\item
			$S$ is a {maximal} Ore set.
		\item
			$S^{-1}R\cong_R\Quot(R)$.
		\item
			$\LSat(S)=R\setminus\set{0}$.
	\end{enumerate} 
\end{prop}

\begin{cor}
	Let $S$ be a left Ore set in a left Ore domain $R$.
	Then the following are equivalent:
	\begin{enumerate}[(1)]
		\item
			$S$ is pre-maximal.
		\item
			$S$ is not maximal and for all non-zero $r\in R\setminus\LSat(S)$, any left Ore set containing both $S$ and $r$ is maximal.
	\end{enumerate}
\end{cor}
\begin{proof}
	\begin{description}
		\implication{1}{2}
			Let $r\in R\setminus\LSat(S)$ and $T$ a left Ore set containing both $S$ and $r$.
			Then $\LSat(S)\subsetneq\LSat(T)$, thus $\LSat(T)$ and therefore $T$ is maximal.
		\implication{2}{1}
			Let $T$ be a left Ore set with $\LSat(S)\subsetneq T$, then $T$ contains a non-zero element $r\in R\setminus\LSat(S)$.
			Since $T$ contains $S$ and $r$, $T$ is maximal.\qedhere
	\end{description}
\end{proof}

The following characterization of saturated multiplicative sets in commutative rings is a classical result that has been part of many exercise sheets:

\begin{lemma}[e.g. Exercise 7 in Chapter 3 in \cite{atiyah-macdonald}]
	Let $S$ be a multiplicative set in a commutative ring $R$.
	Then $S$ is saturated if and only if its complement $R\setminus S$ is a union of prime ideals of $R$.
\end{lemma}

\begin{prop}\label{characterization_of_pre-maximal_m.c._sets}
	Let $S$ be a multiplicative set in a commutative domain $R$.
	Then the following are equivalent:
	\begin{enumerate}[(1)]
		\item
			$S$ is pre-maximal.
		\item
			$\LSat(S)=R\setminus\mf{p}$ for a prime ideal $\mf{p}$ in $R$ of height 1 (in other words, a prime ideal that is minimal among all non-zero prime ideals).
	\end{enumerate}
\end{prop}
\begin{proof}
	Let $S$ be pre-maximal, then $\LSat(S)\subsetneq R\setminus\set{0}$, thus $R\setminus\LSat(S)=\bigcup_{\mf{q}\in P}^{}\mf{q}$, where $P$ is a non-empty set of non-zero prime ideals.
	Let $\mf{p}\in P$, then $T:=R\setminus\bigcup_{\mf{q}\in P\setminus\set{\mf{p}}}^{}\mf{q}$ is a saturated multiplicative set with $\LSat(S)\subsetneq T$, thus $T=R\setminus\set{0}$ which implies $P=\set{\mf{p}}$ and $R\setminus\LSat(S)=\mf{p}$.
	Since for any prime ideal $\mf{q}$ satisfying $\set{0}\subsetneq\mf{q}\subsetneq\mf{p}$ we would have a chain $\LSat(S)=R\setminus\mf{p}\subsetneq R\setminus\mf{q}\subsetneq R\setminus\set{0}$ of saturated multiplicative set in contradiction to $S$ being pre-maximal we have that $\mf{p}$ must be a prime ideal of height 1.
	
	Now let $\mf{p}:=R\setminus\LSat(S)$ be a prime ideal of height 1.
	Then $\LSat(S)=R\setminus\mf{p}\subsetneq R\setminus\set{0}$, thus $S$ is not maximal.
	Let $T$ be a multiplicative set in $R$ such that $R\setminus\mf{p}=\LSat(S)\subsetneq T\subseteq\LSat(T)$.
	In particular, $R\setminus\LSat(T)$ is a union of prime ideals strictly contained in a prime ideal of height one, which implies $R\setminus\LSat(T)=\set{0}$.
	Therefore we have $\LSat(T)=R\setminus\set{0}$ or equivalently that $T$ is maximal.
	Thus, $S$ is pre-maximal.
\end{proof}

\subsection{Localizations of commutative domains}

\begin{lemma}
	Let $A$ and $B$ subsets of a commutative domain such that $W:=\merz{A\cup B}$ is a multiplicative set.
	Then $S:=\merz{A}$ and $T:=\merz{B}$ are multiplicative sets in $R$, while $\rho_{S,R}(T)$ is a multiplicative set in $S\inv R$.
	Furthermore, the map 
	\[
		\varphi:\rho_{S,R}(T)\inv(S\inv R)\rightarrow W\inv R,\quad
		((1,t),(s,r))\mapsto(st,r)
	\]
	is an isomorphism of rings such that $\varphi\circ\rho_{\rho_{S,R}(T),S\inv R}\circ\rho_{S,R}=\rho_{W,R}$.
\end{lemma}
\begin{proof}
	Elementary calculations in commutative localizations.
\end{proof}

\begin{cor}
	Let $S$ be a multiplicative set in a commutative domain $R$ and $p\in R\setminus\set{0}$.
	Then
	\[
		\varphi:\merz*{\frac{p}{1}}\inv(S\inv R)
		=\merz{\rho_{S,R}(p)}\inv(S\inv R)
		\rightarrow\merz{S\cup\set{p}}\inv R,\quad
		((1,p^k),(s,r))\mapsto(sp^k,r)
	\]
	is an isomorphism of rings such that $\varphi\circ\rho_{\merz{\frac{p}{1}},S\inv R}\circ\rho_{S,R}=\rho_{\merz{S\cup\set{p}},R}$.
\end{cor}

\begin{rem}
	This means there are two equivalent ways of enlarging a localization by a single element: assume we are given a commutative domain $R$ and a multiplicative subset $S$.
	Furthermore, there is an element $p\in R$ that should additionally become invertible.
	Then it does not matter if we localize $R$ at $\merz{S\cup\set{p}}$ or if we localize $R$ at $S$ and then localize the result again at $\merz{\rho_{S,R}(p)}$.
\end{rem}

\begin{lemma}\label{ideals_induce_maximal_m.c._sets}
	Let $I$ be a non-zero ideal in a commutative domain $R$.
	Then $\hat{I}:=(I\setminus\set{0})\cup\set{1}$ is a maximal multiplicative set.
\end{lemma}
\begin{proof}
	Let $a,b\in\hat{I}$.
	If $a=1$, then $ab=b\in\hat{I}$.
	If $a\neq1$, then $a\in I$ and thus $ab\in I$.
	Since $R$ is a domain we have $ab\neq0$ and thus $ab\in\hat{I}$.
	Let $f\in I\setminus\set{0}$.
	For all $r\in R\setminus\set{0}$ we have $fr\in I\setminus\set{0}$ and thus $r\in\LSat(\hat{I})$, which implies the maximality of $\hat{I}$ by \Cref{characterization_of_maximal_Ore_sets}.
\end{proof}

\subsection{Localizations of the integers}

The prime numbers $\IP$ are a set of representatives of irreducible elements in the UFD $\IZ$.
Thus the saturated multiplicative sets in $\IZ$ can be identified with subsets of $\IP$ by determining which primes become units in the localization.
A representation of all saturated localizations of $\IZ$ as a binary tree is sketched in \Cref{fig_integer_localizations_binary_tree}.

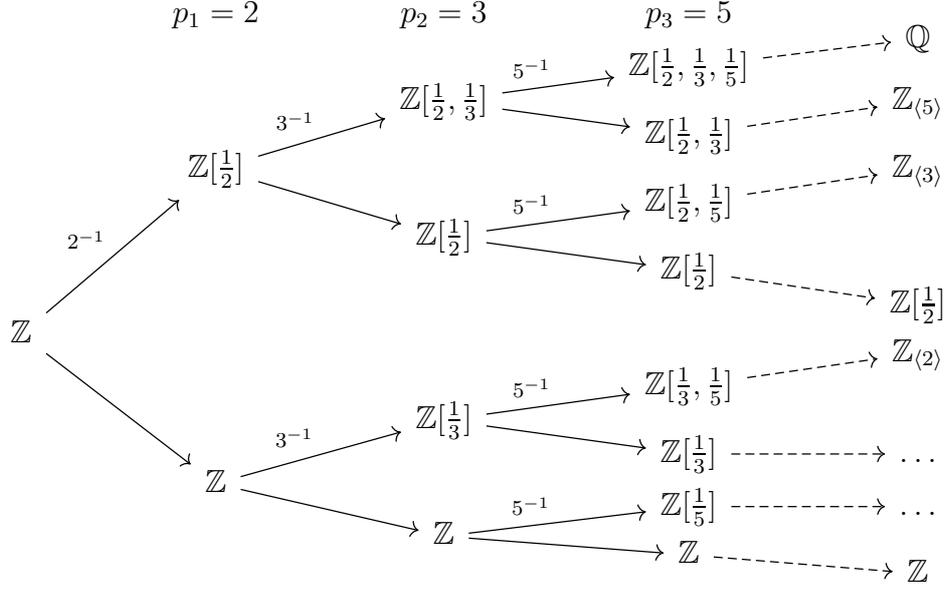
\begin{figure}[ht]
	\[\begin{tikzcd}[row sep = -10pt, column sep = large]
		&p_1=2&p_2=3&p_3=5\\
		&&&&\IQ\\
		&&&\IZ[\frac{1}{2},\frac{1}{3},\frac{1}{5}]\arrow[dashed]{ru}{}\\
		&&\IZ[\frac{1}{2},\frac{1}{3}]\arrow[]{ru}{5\inv}\arrow[]{rd}{}&&\IZ_{\erz{5}}\\
		&&&\IZ[\frac{1}{2},\frac{1}{3}]\arrow[dashed]{ru}{}\\
		&\IZ[\frac{1}{2}]\arrow[]{ruu}{3\inv}\arrow[]{rdd}{}&&&\IZ_{\erz{3}}\\
		&&&\IZ[\frac{1}{2},\frac{1}{5}]\arrow[dashed]{ru}{}\\
		&&\IZ[\frac{1}{2}]\arrow[]{ru}{5\inv}\arrow[]{rd}{}\\
		&&&\IZ[\frac{1}{2}]\arrow[dashed]{rd}{}\\
		&&&&\IZ[\frac{1}{2}]\\
		\IZ\arrow[]{ruuuuu}{2\inv}\arrow[]{rddddd}{}\\
		&&&&\IZ_{\erz{2}}\\
		&&&\IZ[\frac{1}{3},\frac{1}{5}]\arrow[dashed]{ru}{}\\
		&&\IZ[\frac{1}{3}]\arrow[]{ru}{5\inv}\arrow[]{rd}{}\\
		&&&\IZ[\frac{1}{3}]\arrow[dashed]{r}{}&\ldots\\
		&\IZ\arrow[]{ruu}{3\inv}\arrow[]{rdd}{}\\
		&&&\IZ[\frac{1}{5}]\arrow[dashed]{r}{}&\ldots\\
		&&\IZ\arrow[]{ru}{5\inv}\arrow[]{rd}{}\\
		&&&\IZ\arrow[dashed]{rd}{}\\
		&&&&\IZ
	\end{tikzcd}\]
	\caption{Localizations of $\IZ$ at saturated left Ore sets as a binary tree, where each arrow represents an embedding.
	Each level $i$ corresponds to either making the $i$-th prime invertible or not.}
	\label{fig_integer_localizations_binary_tree}
\end{figure}

Furthermore, from \Cref{characterization_of_R-fixing_homomorphisms_via_LSat} we can see that the saturated localizations of $\IZ$ form a bounded lattice with respect to inclusion, where $\IZ$ is the minimal element and $\IQ$ is the maximal element.
A visualization of a small part of this lattice is given in \Cref{fig_integer_localizations_lattice}.

\begin{figure}[ht]
	\[\begin{tikzcd}[row sep = 0, column sep = large]
		&&\IZ[\frac{1}{2},\frac{1}{5}]\arrow[]{rdd}{3\inv}\\
		&\IZ[\frac{1}{2}]\arrow[]{ru}{5\inv}\arrow[]{rd}{3\inv}\\
		&&\IZ[\frac{1}{2},\frac{1}{3}]\arrow[]{r}{5\inv}&\IZ[\frac{1}{2},\frac{1}{3},\frac{1}{5}]\arrow[dashed]{rd}{}\\
		\IZ\arrow[]{ruu}{2\inv}\arrow[]{r}{3\inv}\arrow[]{rdd}{5\inv}&\IZ[\frac{1}{3}]\arrow[]{ru}{2\inv}\arrow[]{rd}{5\inv}&&&\IQ\\
		&&\IZ[\frac{1}{3},\frac{1}{5}]\arrow[]{ruu}{2\inv}\arrow[]{r}{7\inv}&\IZ[\frac{1}{3},\frac{1}{5},\frac{1}{7}]\arrow[dashed]{ru}{}\\
		&\IZ[\frac{1}{5}]\arrow[]{ru}{3\inv}\\
	\end{tikzcd}\]
	\caption{Localizations of $\IZ$ at saturated left Ore sets as a bounded lattice with minimum $\IZ$ and maximum $\IQ$, where the partial order is given by embeddings.}
	\label{fig_integer_localizations_lattice}
\end{figure}
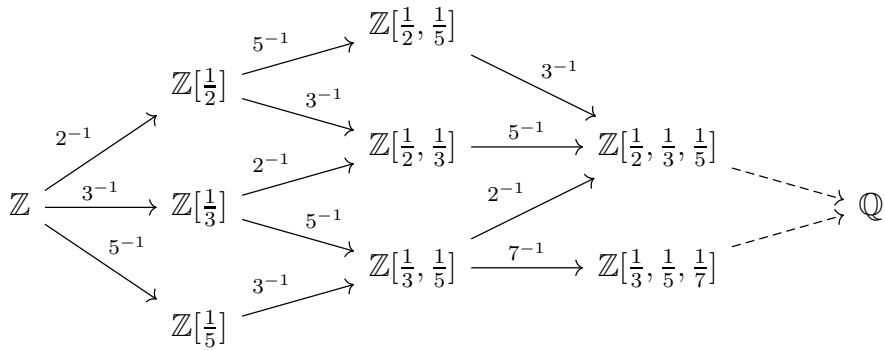

Note that any saturated localization of $\IZ$, which is in bijection with a subset $P\subseteq\IP$, belongs to exactly one of the following three types:
\begin{enumerate}[(i)]
	\item
		$P$ is finite and thus $\IP\setminus P$ is infinite, meaning that only finitely many primes are invertible in the localization.
		\begin{description}
			\item[Special case $P=\emptyset$:]
				This corresponds to the multiplicative set $\units{\IZ}=\set{1,-1}$, which induces the largest localization of $\IZ$ which is still isomorphic to $\IZ$.
			\item[Special case $\abs{P}=1$:]
				These are exactly the smallest non-trivial saturated localizations of $\IZ$, which are given by $\IZ[\frac{1}{p}]:=[p,-1]\inv\IZ\cong_\IZ[p]\inv\IZ$, where $p\in\IP$.
		\end{description}
	\item
		$P$ and $\IP\setminus P$ are infinite.
	\item
		$P$ is infinite and $\IP\setminus P$ is finite, meaning that only finitely many primes are not invertible in the localization.
		\begin{description}
			\item[Special case $\abs{\IP\setminus P}=1$:]
				These are exactly the localizations at pre-maximal multiplicative sets, which by \Cref{characterization_of_pre-maximal_m.c._sets} are given by $\IZ\setminus p\IZ$ for $p\in\IP$.
				The corresponding localizations are thus of the form $\IZ_{\erz{p}}:=(\IZ\setminus p\IZ)\inv\IZ$.
			\item[Special case $P=\IP$:]
				Here we find the localizations at the maximal multiplicative sets in $\IZ$, which include the sets $(n\IZ\setminus\set{0})\cup\set{1}$ where $n\in\IN$.
		\end{description}
\end{enumerate}

Summing up, our machinery shows that the behavior of any localization of the integers depends up to $\IZ$-fixing isomorphisms only on which prime number becomes invertible.

\subsection{Localizations of polynomial rings}

Consider $K[x_1,\ldots,x_n]$ for $n\geq 1$. Recall, that the Gel'fand-Kirillov dimension is considered 
relative to a fixed field, in the rest of our paper we take it to be $K$. 

The crucial difference in the behavior of the Gel'fand-Kirillov dimension over commutative rings to the behavior of the Krull dimension is the following property: for a commutative ring $R$ we have
\[
	\GKdim(R)
	=\sup\set{\GKdim(T)\mid T\text{ is a finitely generated subalgebra of }R}.
\]
Therefore the Gel'fand-Kirillov dimension of a localized commutative ring does not decrease, in particular for a commutative domain and $K$-algebra $R$, $\GKdim(R)=\trdeg_K\Quot(R)$, hence $\GKdim(K)=0$, $\GKdim(\Quot(K[x_1,\ldots,x_n]))=\GKdim(K(x_1,\ldots,x_n))=n$ and thus for any multiplicative set $S$, $\GKdim(S\inv K[x_1,\ldots,x_n])=n$.

The \emph{pre-maximal} sets of $K[x]:=K[x_1,\ldots,x_n]$ are described easily due to \Cref{characterization_of_pre-maximal_m.c._sets}: they are precisely of the form
$K[x]\setminus\mathfrak{p}$, where $\mathfrak{p}$ is a minimal non-zero prime ideal of $K[x]$.

As for \emph{maximal} sets, we have a collection of sets created from taking an ideal and replacing $0$ with $1$ via \Cref{ideals_induce_maximal_m.c._sets}, similarly to the case of $R=\IZ$.
But now there are clearly many more maximal ones: \Cref{BorhoKraft} tells us that
for arbitrary $K$-subalgebra $T\subseteq K[x_1,\ldots,x_n]$,  $\GKdim(T)=n$ implies that $T\setminus\set{0}$ is a maximal Ore set.
Thus for $n=1$ every subalgebra of $K[x]$ deprived of zero except $K$ leads to a maximal Ore set.

\section{Local closure}

While Ore localization of a left module $M$ over a domain $R$ at a left Ore set $S$ can be defined similar to the construction for $S\inv R$ outlined in \Cref{thm_construction}, we introduce the construction in an equivalent but shorter way via a tensor product construction which was inspired by \cite{skoda_2006}:

\begin{definition}
	Let $S$ be a left Ore set in a domain $R$ and $M$ a left $R$-module.
	The \emph{left Ore localization} of $M$ at $S$ is the left $S\inv R$-module $S\inv M:=S\inv R\tensor_RM$.
\end{definition}

Thankfully, when considering elements of $S\inv M$ we only have to look at elementary tensors:

\begin{lemma}[7.3 in \cite{skoda_2006}]
	Let $S$ be a left Ore set in a domain $R$ and $M$ a left $R$-module.
	Every element of $S\inv M$ has a representation $(s,1)\tensor m$ for some $s\in S$ and $m\in M$.
\end{lemma}

Therefore, $S\inv M$ consists of exactly the elements $(s,m):=(s,1)\tensor m$, where $s\in S$ and $m\in M$.
Furthermore, we have $(s,m)=(s,1)\cdot(1,m)$ by the action of $S\inv R$ on $S\inv M$.

Analogously to the localization map $\rho_{S,R}:R\rightarrow S\inv R,~r\mapsto(1,r)$, which embeds a domain $R$ canonically into its localization at $S$, we have the \emph{localization map} for modules:
\[
	\varepsilon_{S,R,M}:M\rightarrow S\inv M,\quad
	m\mapsto(1,1)\tensor m.
\]
This map is a homomorphism of left $R$-modules that is compatible with $\rho_{S,R}$ in the sense that $\varepsilon_{S,R,M}(rm)=\rho_{S,R}(r)\cdot\varepsilon_{S,R,M}(m)$ for all $r\in R$ and $m\in M$.

\begin{rem}
	Let $S$ be a left Ore set in a domain $R$ and $\varphi:M\rightarrow N$ a homomorphism of left $R$-modules.
	Then $S\inv\cdot:=S\inv R\tensor_R\cdot$ becomes an exact covariant functor from the category of left $R$-modules to the category of left $S\inv R$-modules, mapping $M$ to $S\inv M$ and $\varphi$ to
	\[
		S\inv\varphi:S\inv M\rightarrow S\inv N,\quad
		(s,m)\mapsto(s,\varphi(m)).
	\]
\end{rem}

With this notion in mind we turn to another instance of left saturation:

\begin{definition}\label{definition_LSat_in_modules}
	Let $S$ be a left Ore set in a domain $R$ and $P$ a left $R$-submodule of a left $R$-module $M$.
	The $S$-\emph{closure} of $P$ or \emph{local closure} of $P$ at $S$ is defined as
	\[
		P^S
		:=\LSat_S(P)
		=\set{m\in M\mid\exists~s\in S:sm\in P}\supseteq P.
	\]
\end{definition}

From \Cref{definition_LSat_in_ring} and \Cref{definition_LSat_in_modules} we get the following:

\begin{lemma}\label{local_closure_is_independent_of_Ore_representative}
	Let $S$ be a left Ore set in a domain $R$ and $P$ a left $R$-submodule of a left $R$-module $M$.
	Then $P^S=P^{\LSat(S)}$.
\end{lemma}

There is a strong connection between local closure and the extension-contraction problem with respect to the localization map of the associated localization at $S$.

\begin{definition}\label{definition_extension_contraction_modules}
	Let $R,T$ be rings and $M$ a left $R$-module.
	Additionally, let $\varphi:M\rightarrow N$ a homomorphism of left $R$-modules, where $N$ is also a left $T$-module.
	\begin{itemize}
		\item
			The \emph{extension} of a left $R$-submodule $P$ of $M$ to $T$ with respect to $\varphi$ is $P^e:=\erz{\varphi(P)}$, which is the left $T$-submodule of $N$ that is generated by $\varphi(P)$.
		\item
			The \emph{contraction} of a left $T$-submodule $Q$ of $N$ with respect to $\varphi$ is $Q^c:=\varphi\inv(Q)$, which is a left $R$-submodule of $M$.
	\end{itemize}
\end{definition}

The following result is classical:

\begin{lemma}
	In the situation of \Cref{definition_extension_contraction_modules} we have $P\subseteq(P^e)^c$ and $(Q^c)^e\subseteq Q$.
\end{lemma}
\begin{proof}
	We have $P\subseteq\varphi\inv(\varphi(P))\subseteq\varphi\inv(P^e)=(P^e)^c$ and $(Q^c)^e=\erz{\varphi(\varphi\inv(Q))}\subseteq\erz{Q}=Q$.
\end{proof}

Now we consider extension and contraction with respect to the localization map:

\begin{lemma}
	Let $S$ be a left Ore set in a domain $R$, $M$ a left $R$-module and $\varepsilon:=\varepsilon_{S,R,M}$.
	\begin{enumerate}[(a)]
		\item
			Let $Q$ be a left $S\inv R$-submodule of $S\inv M$, then $Q=(Q^c)^e$ with respect to $\varepsilon$.
			In particular, any left $S\inv R$-submodule of $S\inv M$ is the extension of a left $R$-submodule of $M$ with respect to $\varepsilon$.
		\item
			Let $P$ be a left $R$-submodule of $M$, then $P^e=S\inv P$ and $(P^e)^c=P^S$ with respect to $\varepsilon$.
			In particular, $P^S$ is a left $R$-submodule of $M$.
	\end{enumerate}
\end{lemma}
\begin{proof}
	\begin{enumerate}[(a)]
		\item
			Let $(s,m)\in Q$, then $\varepsilon(m)=(1,m)=(1,s)\cdot(s,m)\in Q$, thus $m\in\varepsilon\inv(Q)=Q^c$, which implies $\varepsilon(m)\in\varepsilon(Q^c)$.
			Now $(s,m)=(s,1)\cdot(1,m)=(s,1)\cdot\varepsilon(m)\in\erz{\varepsilon(Q^c)}=(Q^c)^e$.
		\item
			Let $(s,p)\in S\inv P$, then $(s,p)=(s,1)\cdot(1,p)=(s,1)\cdot\varepsilon(p)\in\erz{\varepsilon(P)}=P^e$.
			On the other hand, let $x\in P^e$, then $x=\sum_{i=1}^{n}(s_i,r_i)\cdot\ve(p_i)$ for some $s_i\in S$, $r_i\in R$ and $p_i\in P$.
			By the left Ore condition on $S$ there exist $a_i\in R$ and $s\in S$ such that $a_is_i=s$ for all $i$, then
			\[
				x
				=\sum_{i=1}^{n}(s_i,r_i)\cdot\ve(p_i)
				=\sum_{i=1}^{n}(a_is_i,a_ir_i)\cdot(1,p_i)
				=\sum_{i=1}^{n}(s,a_ir_ip_i)
				=\left(s,\sum_{i=1}^{n}a_ir_ip_i\right)
				\in S\inv P,
			\]
			which shows $P^e=S\inv P$.\\
			Let $m\in(P^e)^c$, then $\varepsilon(m)\in P^e=S\inv P$, thus there exist $s\in S$ and $p\in P$ such that $(1,m)=\varepsilon(m)=(s,p)$.
			This in turn implies the existence of $\tilde{s}\in S$ and $\tilde{r}\in R$ such that $\tilde{s}=\tilde{s}\cdot1=\tilde{r}s$ and $\tilde{s}m=\tilde{r}p\in P$, which implies $m\in\LSat_S(P)=P^S$.\\
			Lastly, let $m\in P^S$, then there exists $s\in S$ such that $sm\in P$.
			Now $\varepsilon(m)=(1,m)=(s,sm)\in S\inv P=P^e$, thus $m\in\varepsilon\inv(P^e)=(P^e)^c$.\qedhere
	\end{enumerate}
\end{proof}

The local closure of submodules and ideals is an important object in algebraic analysis, but its computation is notoriously hard.
The multivariate partial differential case corresponding to the $n$-th Weyl algebra over a field and $S=K[x]\setminus\set{0}$ has been thoroughly investigated by Tsai in \cite{tsai_weyl_closure}, culminating in the Weyl closure algorithm.
More recently there have been several papers studying closure in other univariate algebras of operators, but a general algorithm applying to the multivariate case as in the Weyl situation is still unknown.

An important special case of local closure is local torsion, which is the local closure of the trivial submodule:

\begin{definition}
	Let $S$ be a left Ore set in a domain $R$ and $M$ a left $R$-module.
	Then the $S$-\emph{torsion} submodule of $M$ is
	\[
		t_S(M)
		:=\ker(\varepsilon_{S,R,M})
		=\set{m\in M\mid\exists s\in S:sm=0}
		=\LSat_S^M(\set{0}).
	\]
	The module $M$ is called $S$-\emph{torsion-free} if $t_S(M)=\set{0}$ and $S$-\emph{torsion (module)} if $t_S(M)=M$.
\end{definition}

\begin{prop}
	Let $S$ be a left Ore set in a domain $R$ and $M$ a left $R$-module.
	\begin{enumerate}[(a)]
		\item
			$t_S(\cdot)$ is a covariant left-exact functor from the category of left $R$-modules to the category of left $S$-torsion $R$-modules.
			It is functorially isomorphic to $\Tor_1^R(S\inv R/R,\cdot)$.
		\item
			$t_S(M)=t_{\LSat(S)}(M)$.
		\item
			If $M$ is finitely presented via $M\cong R^n/P$, then $t_S(M)=P^S/P$.
			In particular, $M$ is $S$-torsion-free if and only if $P$ is left $S$-saturated.
	\end{enumerate}
\end{prop}
\begin{proof}
	\begin{enumerate}[(a)]
		\item
			Since $0\rightarrow R\rightarrow S\inv R\rightarrow S\inv R/R\rightarrow0$ is an exact sequence of $(R,R)$-bimodules, applying the right exact functor $\cdot\tensor_RM$ yields the exact sequence
			\[
				R\tensor_RM\cong M
				\rightarrow S\inv M
				\rightarrow(S\inv R/R)\tensor_RM
				\rightarrow0.
			\]
			Since $\Tor$ is the left derived functor of the tensor product functor, this sequence can be extended to a longer exact sequence as follows:
			\[
				\Tor_1^R(R,M)
				\rightarrow\Tor_1^R(S\inv R,M)
				\rightarrow\Tor_1^R(S\inv R/R,M)
				\rightarrow M
				\rightarrow S\inv M
				\rightarrow\ldots
			\]
			Now $\Tor_1^R(R,M)=\Tor_1^R(S\inv R,M)=0$ since $R$ is a free $R$-module and $S\inv R$ is flat.
			Thus
			\[
				0
				\rightarrow\Tor_1^R(S\inv R/R,M)
				\rightarrow M
				\rightarrow S\inv M
			\]
			is exact, which implies $\Tor_1^R(S\inv R/R,M)\cong\ker(\varepsilon_{S,R,M})=t_S(M)$.
		\item
			With \Cref{local_closure_is_independent_of_Ore_representative} we get $t_S(M)=\set{0}^S=\set{0}^{\LSat(S)}=t_{\LSat(S)}(M)$.
		\item
			For any $x\in P^S$ with $s\in S$ such that $sx\in P$ we have $s(x+P)=sx+sP\subseteq P$, thus $P^S/P\subseteq t_S(M)$.
			On the other hand, for any $x+P\in t_S(M)$ with $s\in S$ such that $s(x+P)\subseteq P$ we have $sx\in P$ and thus $t_S(M)\subseteq P^S/P$.\qedhere
	\end{enumerate}
\end{proof}

\section{Iterated closures}

In this section we describe how to split the complicated problem of computing a closure with respect to $\merz{S\cup T}$ into several possibly easier instances of computing closures with respect to $S$ and $T$.

\begin{lemma}\label{iterated_closure}
	Let $S$ be a multiplicative set in a domain $R$, $M$ a left $R$-module and $P$ a left $R$-submodule of $M$.
	Further, let $\set{S_i}_{i\in I}$ be a family of multiplicative sets in $R$ such that $S=\merz{\bigcup_{i\in I}^{}S_i}$.
	Then $P$ is left $S$-saturated if and only if $P$ is left $S_i$-saturated for all $i\in I$.
\end{lemma}
\begin{proof}
	If $P=P^S$, then $P\subseteq P^{S_i}\subseteq P^S=P$ implies $P=P^{S_i}$ for all $i$.
	Now let $P=P^{S_i}$ for all $i$ and $m\in P^S$, then there exists $w\in S$ such that $wm\in P$.
	By definition, $w=w_{i_1}w_{i_2}\cdot\ldots\cdot w_{i_k}$ for some $i_j\in I$ and $w_{i_j}\in S_{i_j}$.
	Since $P=P^{S_{i_1}}$, from $w_{i_1}w_{i_2}\cdot\ldots\cdot w_{i_k}m\in P$ we get $w_{i_2}\cdot\ldots\cdot w_{i_k}m\in P$.
	By induction we get $m\in P$ and thus $P=P^S$.
\end{proof}

\begin{cor}
	Let $S$ be a multiplicative set in a domain $R$, $M$ a left $R$-module and $P$ a left $R$-submodule of $M$.
	Further, let $S_1,\ldots,S_k$ be multiplicative sets in $R$ such that $S=\merz{\bigcup_{i=1}^{k}S_i}$ and $S_iS_j=S_jS_i$ for all $i,j$.
	Define $P_0:=P$ and $P_i:=P_{i-1}^{S_i}$ for $i\in\nset{k}$, then $P^S=P_k$.
\end{cor}
\begin{proof}
	If $S_iS_j=S_jS_i$ then $\LSat_{S_i}(\LSat_{S_j}(P))=\LSat_{S_j}(\LSat_{S_i}(P))$, which is contained in $P^S$ as well as both left $S_i$-saturated and left $S_j$-saturated.
	The statement then follows by induction.
\end{proof}

In order to compute the closure with respect to a union of non-commutating left Ore sets, much more work needs to be done:

\begin{prop}\label{nc_iterated_closure}
	Let $R$ be a domain, $M$ a left $R$-module and $P$ a left $R$-submodule of $M$.
	Further, let $\set{S_i}_{i\in I}$ be a countable family of left Ore sets in $R$, $S:=\merz{\bigcup_{i\in I}^{}S_i}$ and $f:\IN\rightarrow I$ an indexing map.
	Define $P_0:=P$ and $P_n:=\LSat_{S_{f(n)}}(P_{n-1})$ for all $n\in\IN$.
	\begin{enumerate}[(a)]
		\item
			For all $n\in\IN$, $P_n=\LSat_{S_{f(1)}S_{f(2)}\cdot\ldots\cdot S_{f(n)}}(P)$.
		\item
			For all $n\in\IN$, 
			$\LSat_S(P_n)=\LSat_S(P)$.
		\item
			Let $\ell,m\in\IN_0$ such that $\set{f(\ell+j)\mid j\in\nset{m}}=I$ and $P_\ell=P_{\ell+m}$, then $P_\ell=\LSat_S(P)$.
		\item
			Let $I=\nset{k}$, $S_iS_j=S_jS_i$ for all $i,j\in I$ and $f(I)=I$, then $P_k=\LSat_S(P)$.
	\end{enumerate}
\end{prop}
\begin{proof}
	\begin{enumerate}[(a)]
		\item
			Follows via induction, since by \Cref{basic_properties_of_general_LSat} we have
			\[
				P_2
				=\LSat_{S_{f(2)}}(P_1)
				=\LSat_{S_{f(2)}}(\LSat_{S_{f(1)}}(P))
				=\LSat_{S_{f(1)}S_{f(2)}}(P).
			\]
		\item
			Since $S_{f(1)}S_{f(2)}\cdot\ldots\cdot S_{f(n)}\subseteq S$ we have
			\[
				P
				\subseteq P_n
				=\LSat_{S_{f(1)}S_{f(2)}\cdot\ldots\cdot S_{f(n)}}(P)
				\subseteq\LSat_S(P)
			\]
			and thus $\LSat_S(P_n)=\LSat_S(P)$ by \Cref{sufficient_conditions_for_equality_of_closures}.
		\item
			By assumption, for any $i\in I$ there is a $j\in\nset{m}$ such that $f(\ell+j)=i$.
			Then $P_\ell\subseteq P_{\ell+j}\subseteq P_{\ell+m}=P_\ell$, thus $P_\ell=P_{\ell+j}=\LSat_{S_i}(P_{\ell+j-1})$ is left $S_i$-saturated.
			By \Cref{iterated_closure} we have $P_\ell=\LSat_S(P_\ell)=\LSat_S(P)$.
		\item
			If $S_iS_j=S_jS_i$, then
			\[
				\LSat_{S_i}(\LSat_{S_j}(P))
				=\LSat_{S_jS_i}(P)
				=\LSat_{S_iS_j}(P)
				=\LSat_{S_j}(\LSat_{S_i}(P))
			\]
			by \Cref{basic_properties_of_general_LSat}, so by induction we have that $P_k$ is left $S_i$-saturated for all $i\in I$.
			From \Cref{iterated_closure} we get $P_k=\LSat_S(P_k)=\LSat_S(P)$.\qedhere
	\end{enumerate}
\end{proof}

Note that the ``termination condition'' given in part (c) can only be satisfied in the case of a finite index set $I$.

\begin{ex}
	Let $S_x:=\merz{x}$, $S_{\partial}:=\merz{\partial}$ and $S:=\merz{S_x\cup S_{\partial}}$ in $\mc{D}$. In \Cref{Theta_is_left_Ore} we have seen that $\LSat(S)=\merz{(\theta+\IZ)\cup\set{x,\partial}}$.
	The theory, developed by Tsai \cite{tsai_weyl_closure} applies to the closure of a module with respect to $S_x$ and, as we will see below, also with respect to $S_{\partial}$.
	But no method, known up to now, is able to compute the closure with respect to $S$.
	This is now possible thanks to \Cref{nc_iterated_closure}.
	
	Let $L$ be the left ideal generated by $\partial\cdot(x\partial+3)\cdot(3x\partial+1)\cdot(x+\partial)$. From the considerations above, it is clear that $\erz{(3x\partial+1)\cdot(x+\partial)}\subseteq L^S$.
	In order to prove the equality we need to perform computations.
	Let us denote by $\tau:\mc{D}\rightarrow\mc{D}$ the automorphism defined by $x\mapsto-\d$ and $\d\mapsto x$, which is called the \emph{Fourier transform}.
	
	By \cite{tsai_weyl_closure} we can compute closures with respect to $S_x$, and concrete computations show that $L^{S_x}=L$, even though the leading term of the generator is $3x^2\cdot \d^4$.
	Then, since $S_{\partial} = \tau(S_x)$, we obtain with further computations that 
	\[
		L_1
		:=L^{S_{\partial}}
		=\tau({\tau(L)}^{S_x})
		=\erz{
			3x^{4}\d^{2}+(3x^{5}+x^{3})\d+4x^{4},
			3x^{2}\d^{3}+(3x^{3}+13x)\d^{2}+(19x^{2}+3)\d+16x
		}.
	\]
	Then $L_2 := L_1^{S_x} = \langle 3x\d^2+(3x^2+1)\d+4x \rangle = \langle (3x\d+1)(\d+x) \rangle$
	and $L_2^{S_{\partial}} = L_2$ shows that $L_2$ is the closure $L^S$.
\end{ex}

This approach of breaking denominator sets into more easily manageable subsets motivates the following partial classification of standard building blocks, which are inspired by the most common localizations in the world of commutative algebra.

\begin{definition}\label{localization_types}
	Let $S$ be a left denominator set in a ring $R$.
	Then $S$ (and by extension, the localization $S\inv R$) might belong to one (or multiple) of the following types.
	\begin{description}
		\item[Monoidal:]
			$S$ is generated as a multiplicative monoid by at most countably many elements.
		\item[Geometric:]
			$S=T\setminus\mf{p}$ for a commutative subring $T$ of $R$ with $\mf{p}$ being a prime ideal in $T$.
		\item[Rational:]
			$S\cup\set{0}$ is a subring of $R$.
	\end{description}
\end{definition}

\begin{ex}\label{example_types}
	Consider the following localizations of the first Weyl algebra $\mc{D}$, which are all subrings of the rational localization $\Quot(\mc{D})=(\mc{D}\setminus\set{0})\inv\mc{D}$:
	\begin{itemize}
		\item
			Taking $S=\merz{x}$ leads to the monoidal localization
			\[
				\merz{x}\inv\mc{D}
				\cong K\erz{x,x\inv,\partial\mid\partial x=x\partial+1,\partial x\inv=x\inv\partial+x^{-2},xx\inv=x\inv x=1},
			\]
			which can be seen as a Weyl algebra with Laurent polynomial coefficients.
		\item
			Let $p\in K$ and consider the maximal ideal $\mf{m}_p:=\erz{x-p}$ in the commutative subring $K[x]$ of $\mc{D}$.
			The geometric localization
			\[
				(K[x]\setminus\mf{m}_p)\inv\mc{D}
				\cong K[x]_{\mf{m}_p}\erz{\partial\mid\partial f=f\partial+\frac{\text{d}f}{\text{d}x}\text{ for all }f\in K[x]_{\mf{m}_p}\subseteq K(x)}
			\]
			is a so-called \emph{local (algebraic) Weyl algebra}, which is of importance in $D$-module theory.
		\item
			Rational localization at $K[x]\setminus\set{0}$ allows us to pass from the polynomial Weyl algebra $\mc{D}$ to the rational Weyl algebra
			\[
				(K[x]\setminus\set{0})\inv\mc{D}
				\cong K(x)\erz*{\partial\mid\partial f=f\partial+\frac{\text{d}f}{\text{d}x}\text{ for all }f\in K(x)}.
			\]
	\end{itemize}
	Note, that all of algebras above are of Gel'fand-Kirillov dimension 2 over any field.
	These examples can be easily generalized for multivariate Weyl algebras.
\end{ex}

An implementation of basic arithmetic in Ore-localized $G$-algebras is available for certain special cases of the types defined in \Cref{localization_types} in the library \texttt{olga.lib} for the computer algebra system \textsc{Singular:Plural} (\cite{Plural}), including the localizations in \Cref{example_types}.
Details can be found in \cite{HL_ISSAC17}.

\section{Conclusion}

The theory of \emph{left saturation} developed in this paper yields insight into two important applications of Ore localizations that seem to have little in common at first glance.

Applied to a left Ore set $S$ in a (possibly non-commutative) domain $R$ it serves as the \emph{canonical description} of the localization $S\inv R$ up to a unique isomorphism and thus reveals the underlying structure: $\LSat(S)$ is a saturated left Ore superset of $S$ that describes all units in $S\inv R$ such that $S\inv R$ is canonically isomorphic to $\LSat(S)\inv R$.
We have shown at some classical commutative examples how this machinery can be used to classify all localizations of a given domain.

Furthermore, left saturation sheds new light onto the problem of \emph{local closure}, i.e. of localizing a submodule and then contracting it back again.
We have given a general algorithmic approach to split the computation of a local closure into several easier computations, which allows us to tackle important problems that were not computable before.

\section{Acknowledgements}

The authors are grateful to Vladimir Bavula (Sheffield) and Eva Zerz (Aachen) for fruitful discussions.

The second author has been supported by Project II.6 of SFB-TRR 195 ``Symbolic Tools in Mathematics and their Applications'' of the German Research Foundation (DFG).

Both authors are partially supported by Research Training Group 1632 ``Experimental and Constructive Algebra'' of the DFG.

\newpage
\appendix

\section{Proof of \texorpdfstring{\Cref{omega_lemma}}{Lemma 6.2}}\label{appendix_omega_lemma}

\omegalemma*
\begin{proof}
	Let $(s,r),(s_1,r_1),(s_2,r_2)\in S\inv R$.
	\begin{description}
		\item[Independence of the choice of the $w_s$:]
			Let $(s_1,r_1)=(s_2,r_2)$, then there exist $s_a\in S$ and $r_a\in R$ such that
			\begin{align}
				s_as_2
				&=r_as_1\label{eq:mon_module_LSat_incl_well_1},\\
				s_ar_2
				&=r_ar_1\label{eq:mon_module_LSat_incl_well_2}.
			\end{align}
			Furthermore, let $w_1,w_2\in R$ such that $w_1s_1,w_2s_2\in T$.
			We have to show that
			\[
				(w_1s_1,w_1r_1)
				=(w_2s_2,w_2r_2)
			\]
			in $T\inv R$.
			By the left Ore condition on $T$ there exist $t_b\in T$ and $r_b\in R$ such that
			\begin{equation}
				t_bw_2s_2
				=r_bw_1s_1\label{eq:mon_module_LSat_incl_well_3}.
			\end{equation}
			By the left Ore condition on $S$ there exist $s_c\in S$ and $r_c\in R$ such that
			\begin{equation}
				s_ct_bw_2
				=r_cs_a\label{eq:mon_module_LSat_incl_well_4}.
			\end{equation}
			Since $s_c\in S\subseteq\LSat(T)$ there exists $w_c\in R$ such that $w_cs_c\in T$.
			Now
			\[
				w_cs_cr_bw_1s_1
				\overset{\eqref{eq:mon_module_LSat_incl_well_3}}{=}
					w_cs_ct_bw_2s_2
				\overset{\eqref{eq:mon_module_LSat_incl_well_4}}{=}
					w_cr_cs_as_2
				\overset{\eqref{eq:mon_module_LSat_incl_well_1}}{=}
					w_cr_cr_as_1
			\]
			implies
			\begin{equation}
				w_cs_cr_bw_1
				=w_cr_cr_a\label{eq:mon_module_LSat_incl_well_5},
			\end{equation}
			since $s_1\neq0$ and $R$ is a domain.
			Define $\mathring{t}:=w_cs_ct_b\in T$ and $\mathring{r}:=w_cs_cr_b\in R$, then
			\[
				\mathring{t}w_2s_2
				=w_cs_ct_bw_2s_2
				\overset{\eqref{eq:mon_module_LSat_incl_well_3}}{=}
					w_cs_cr_bw_1s_1
				=\mathring{r}w_1s_1
			\]
			and
			\[
				\mathring{t}w_2m_2
				=w_cs_ct_bw_2m_2
				\overset{\eqref{eq:mon_module_LSat_incl_well_4}}{=}
					w_cr_cs_am_2
				\overset{\eqref{eq:mon_module_LSat_incl_well_2}}{=}
					w_cr_cr_am_1
				\overset{\eqref{eq:mon_module_LSat_incl_well_5}}{=}
					w_cs_cr_bw_1m_1
				=\mathring{r}w_1m_1
			\]
			show that $(w_1s_1,w_1m_1)=(w_2s_2,w_2m_2)$.
		\item[Neutral element/$R$-fixing:]
			Let $r\in R$.
			Since $1_R\cdot1_R=1_R\in T$ we have
			\[
				\omega(\rho_{S,R}(r))
				=\omega((1_R,r))
				=(1_R\cdot1_R,1_R\cdot r)
				=(1_R,r)
				=\rho_{T,R}(r).
			\]
			In particular, we have $\omega(1_{S\inv R})=\omega(\rho_{S,R}(1_R))=\rho_{T,R}(1_R)=1_{T\inv R}$.
		\item[Additivity:]
			We have
			\[
				\omega((s_1,r_1)+(s_2,r_2))
				=\omega((s_as_1,s_ar_1+r_ar_2))
				=(w_as_as_1,w_as_ar_1+w_ar_ar_2)
			\]
			and
			\[
				\omega((s_1,r_1))+\omega((s_2,r_2))
				=(w_1s_1,w_1r_1)+(w_2s_2,w_2r_2)
				=(t_bw_1s_1,t_bw_1r_1+r_bw_2r_2),
			\]
			where $s_a\in S$, $r_a,r_b,w_a,w_1,w_2\in R$ and $t_b\in T$ satisfy
			\begin{equation}
				s_as_1
				=r_as_2,\label{eq:LSat_inclusion_add_1}
			\end{equation}
			$w_as_as_1\in T$, $w_1s_1\in T$, $w_2s_2\in T$ and
			\begin{equation}
				t_bw_1s_1
				=r_bw_2s_2.\label{eq:LSat_inclusion_add_2}
			\end{equation}
			We have to show $(w_as_as_1,w_as_ar_1+w_ar_ar_2)=(t_bw_1s_1,t_bw_1r_1+r_bw_2r_2)$.
			By the left Ore condition on $T$ there exist $\mathring{t}\in T$ and $\mathring{r}\in R$ such that
			\[
				\mathring{t}t_bw_1s_1
				=\mathring{r}w_as_as_1,
			\]
			which implies
			\begin{equation}
				\mathring{t}t_bw_1
				=\mathring{r}w_as_a\label{eq:LSat_inclusion_add_4},
			\end{equation}
			since $s_1\neq0$ and $R$ is a domain.
			Now
			\[
				\mathring{t}r_bw_2s_2
				\overset{\eqref{eq:LSat_inclusion_add_2}}{=}\mathring{t}t_bw_1s_1
				\overset{\eqref{eq:LSat_inclusion_add_4}}{=}\mathring{r}w_as_as_1
				\overset{\eqref{eq:LSat_inclusion_add_1}}{=}\mathring{r}w_ar_as_2
			\]
			implies
			\begin{equation}
				\mathring{t}r_bw_2
				=\mathring{r}w_ar_a\label{eq:LSat_inclusion_add_5},
			\end{equation}
			since $s_2\neq0$ and $R$ is a domain.
			Now
			\[\begin{split}
				\mathring{t}(t_bw_1r_1+r_bw_2r_2)
				&=\mathring{t}t_bw_1r_1+\mathring{t}r_bw_2r_2
				\overset{\eqref{eq:LSat_inclusion_add_4}}{=}
					\mathring{r}w_as_ar_1+\mathring{t}r_bw_2r_2\\
				&\overset{\eqref{eq:LSat_inclusion_add_5}}{=}
					\mathring{r}w_as_ar_1+\mathring{r}w_ar_ar_2
				=\mathring{r}(w_as_ar_1+w_ar_ar_2)
			\end{split}\]
			shows that $(w_as_as_1,w_as_am_1+w_ar_am_2)=(t_bw_1s_1,t_bw_1m_1+r_bw_2m_2)$, thus $\omega$ is additive.
		\item[Multiplicativity:]
			We have
			\[
				\omega((s_1,r_1)\cdot(s_2,r_2))
				=\omega((s_as_1,r_a,r_2))
				=(w_as_as_1,w_ar_ar_2)
			\]
			and
			\[
				\omega((s_1,r_1))\cdot\omega((s_2,r_2))
				=(w_1s_1,w_1r_1)\cdot(w_2s_2,w_2r_2)
				=(t_bw_1s_1,r_bw_2r_2)
			\]
			where $s_a\in S$, $r_a,r_b,w_a,w_1,w_2\in R$ and $t_b\in T$ satisfy
			\begin{equation}
				s_ar_1
				=r_as_2,\label{eq:LSat_inclusion_mult_1}
			\end{equation}
			$w_as_as_1\in T$, $w_1s_1\in T$, $w_2s_2\in T$ and
			\begin{equation}
				t_bw_1r_1
				=r_bw_2s_2.\label{eq:LSat_inclusion_mult_2}
			\end{equation}
			We have to show $(w_as_as_1,w_ar_ar_2)=(t_bw_1s_1,r_bw_2r_2)$ in $T\inv R$.
			By the left Ore condition on $T$ there exist $\mathring{t}\in T$ and $\mathring{r}\in R$ such that
			\[
				\mathring{t}t_bw_1s_1
				=\mathring{r}w_as_as_1,
			\]
			which implies
			\begin{equation}
				\mathring{t}t_bw_1
				=\mathring{r}w_as_a\label{eq:LSat_inclusion_mult_4},
			\end{equation}
			since $s_1\neq0$ and $R$ is a domain.
			Now
			\[
				\mathring{t}r_bw_2s_2
				\overset{\eqref{eq:LSat_inclusion_mult_2}}{=}
					\mathring{t}t_bw_1r_1
				\overset{\eqref{eq:LSat_inclusion_mult_4}}{=}
					\mathring{r}w_as_ar_1
				\overset{\eqref{eq:LSat_inclusion_mult_1}}{=}
					\mathring{r}w_ar_as_2
			\]
			implies
			\begin{equation}
				\mathring{t}r_bw_2
				=\mathring{r}w_ar_a\label{eq:LSat_inclusion_mult_5},
			\end{equation}
			since $s_2\neq0$ and $R$ is a domain.
			Now
			\[
				\mathring{t}r_bw_2r_2
				=t_cr_bw_2r_2
				\overset{\eqref{eq:LSat_inclusion_mult_5}}{=}
					r_cw_ar_ar_2
				=\mathring{r}w_ar_ar_2
			\]
			shows that $(w_as_as_1,w_ar_ar_2)=(t_bw_1s_1,r_bw_2r_2)$.
		\item[Injectivity:]
			Let $0=\omega((s,r))=(ws,wr)$, where $w\in R$ satisfies $ws\in T$, in particular, we have $w\neq0$.
			Now $wr=0$ and $w\neq0$ imply $r=0$ and thus $(s,r)=0$, since $R$ is a domain.\qedhere
	\end{description}
\end{proof}

\newpage

\nocite{*}
\bibliography{literatur}
\bibliographystyle{plain}

\end{document}